\pdfoutput=1

\documentclass[10pt]{article}
\usepackage[hmargin=25mm,vmargin=25mm]{geometry}
\usepackage{amsmath, amsthm, amssymb, mathabx, verbatim, setspace, enumerate}
\usepackage{relsize}
\usepackage{url}
\usepackage{hyperref}
\usepackage[mathscr]{euscript}
\usepackage[all,cmtip]{xy}
\usepackage{relsize}
\usepackage[toc,page]{appendix}

\usepackage[utf8]{inputenc}

\usepackage{tikz-cd}

\xyoption{rotate}
\xyoption{pdf}

\usepackage{tikz-cd}

\newtheorem{theorem}{Theorem}

\newtheorem{thm}{Theorem}[subsection]
\newtheorem{prop}[thm]{Proposition}
\newtheorem{lemma}[thm]{Lemma}
\newtheorem{cor}[thm]{Corollary}

\theoremstyle{definition}
\newtheorem{exmp}[thm]{Example}
\newtheorem{defn}[thm]{Definition}

\newtheorem{rmk}[thm]{Remark}

\DeclareMathOperator{\Hom}{Hom}
\DeclareMathOperator{\Sym}{Sym}
\newcommand{\colim}{\operatornamewithlimits{colim}}

\DeclareMathOperator{\PreMF}{PreMF}
\DeclareMathOperator{\Ho}{Ho}
\DeclareMathOperator{\Pro}{Pro}
\DeclareMathOperator{\Ind}{Ind}

\DeclareMathOperator{\dmod}{-mod}

\DeclareMathOperator{\Spec}{Spec}

\DeclareMathOperator{\dperf}{-perf}
\DeclareMathOperator{\Alg}{Alg}

\DeclareMathOperator{\shHom}{\mathcal{H}\emph{om}}

\DeclareMathOperator{\pt}{pt}

\DeclareMathOperator{\Aut}{Aut}

\DeclareMathOperator{\dR}{dR}

\DeclareMathOperator{\Irr}{Irr}

\DeclareMathOperator{\QCoh}{QCoh}
\DeclareMathOperator{\Map}{Map}
\DeclareMathOperator{\Fun}{Fun}

\DeclareMathOperator{\Perf}{Perf}
\DeclareMathOperator{\IndCoh}{IndCoh}

\DeclareMathOperator{\Coalg}{Coalg}

\DeclareMathOperator{\Tate}{Tate}

\newcommand{\mf}{\mathfrak}

\newcommand{\opcit}{\emph{op.\,cit.} }

\newcommand{\bs}{\backslash}
\newcommand{\cat}{\mathbf}

\newcommand{\Z}{\mathbb{Z}}
\newcommand{\C}{\mathbb{C}}

\newcommand{\G}{\mathbb{G}}

\newcommand{\OO}{\mathcal{O}}

\newcommand{\LL}{\mathcal{L}}
\newcommand{\LLf}{\widehat{\LL}}
\newcommand{\h}{\mf{h}}
\newcommand{\g}{\mf{g}}

\DeclareUnicodeCharacter{0301}{\`}

\newcommand{\wh}{\widehat}
\DeclareMathOperator{\cone}{cone}

\newcommand{\hcdR}{\widehat{\mathrm{dR}}}

\newcommand{\HP}{HP}
\DeclareMathOperator{\Coh}{Coh}

\title{Equivariant localization and completion in cyclic homology and derived loop spaces}
\date{}
\author{Harrison Chen}

\newcommand{\HH}{HH}

\begin{document}
\maketitle

\begin{abstract}
We prove an equivariant localization theorem over an algebraically closed field of characteristic zero for smooth quotient stacks by reductive groups $X/G$ in the setting of derived loop spaces as well as Hochschild homology and its cyclic variants.  We show that the derived loop spaces of the stack $X/G$ and its classical $z$-fixed point stack $\pi_0(X^z)/G^z$ become equivalent after completion along a semisimple parameter $[z] \in G//G$, implying the analogous statement for Hochschild and cyclic homology of the dg category of perfect complexes $\Perf(X/G)$.  We then prove an analogue of the Atiyah-Segal completion theorem in the setting of periodic cyclic homology, where the completion of the periodic cyclic homology of $\Perf(X/G)$ at the identity $[e] \in G//G$ is identified with a 2-periodic version of the derived de Rham cohomology of $X/G$.  Together, these results identify the completed periodic cyclic homology of a stack $X/G$ over a parameter $[z] \in G//G$ with the 2-periodic derived de Rham cohomology of its $z$-fixed points.
\end{abstract}

\tableofcontents

\section{Introduction}

Derived loop spaces appear naturally in questions of Hochschild homology.  When $X$ is a prestack such that the derived categories $\QCoh(X)$ and $\IndCoh(X)$ are 
compactly generated by $\Perf(X)$ and $\Coh(X)$ respectively, their Hochschild homology can be computed in two ways.  On one hand, it is computed by the usual cyclic bar complex, but it can also be naturally identified with the global sections of a sheaf of functions or distributions on the derived loop space $\LL X$ of $X$ \cite{BZN:NT}:
$$\HH(\QCoh(X)) = \OO(\LL X), \;\;\;\;\;\;\; \HH(\IndCoh(X)) = \omega(\LL X).$$
When $X$ is a derived scheme, the loop space $\LL X$ is equivalent to the shifted odd tangent bundle via a derived variant of the Hochschild-Kostant-Rosenberg theorem, reflecting the Zariski-local nature of derived loops.  Introducing the $S^1$-action, results in \cite{BZN:LC} \cite{Pr:IL} identify periodic cyclic homology with 2-periodic derived de Rham (co)homology:
$$\HP(\Perf(X)) = \OO(\LL X)^{\Tate} \simeq C^\bullet_{\dR}(X; k)((u)), \;\;\;\; \HP(\Coh(X)) = \omega(\LL X)^{\Tate} \simeq C_{\bullet, \dR}(X; k)((u))$$
where $u \in C^2(BS^1; k)$ is the degree two universal Chern class.  Furthermore, over $k=\C$, the main result in \cite{Bh:DDR} identifies derived de Rham cohomology with Betti de Rham cohomology of the analytification, and we have that $\HP(\Perf(X)) \simeq C^\bullet_{\dR}(X^{an}; k)((u))$.  Our goal in this paper is to investigate the periodic cyclic homology where $X$ is taken to be a smooth quotient stack.

A difficulty in understanding the Hochschild homology of stacks which does not appear for schemes is the failure of derived loop spaces to form a cosheaf for the smooth topology; informally, Hochschild homology is Zariski local but not smooth local\footnote{This depends on definitions.  Our notion of Hochschild homology is the one that is computed by a cyclic bar complex of a small dg  category.  Other authors have considered a variant which is obtained by extending the notion of Hochschild homology of affine schemes to stacks via flat descent, which is smooth local by definition and is equivalent to the a completed version of the former.}.  On the other hand, as argued in \cite{BZN:LC}, sheaves and functions on formal loop spaces (i.e. loop spaces completed at constant loops) can be computed via a smooth covers.  We begin to bridge this gap by understanding the global loop space $\LL(X/G)$ after completion over fibers of the map to $\LL(BG) = G/G$ as well as to its affinization $G//G$.  We show that completion over $G/G$ gives a certain formal loop space and completion over $G//G$ gives a certain unipotent loop space, formalizing a ``Jordan decomposition for loop spaces'' which appeared in \cite{BZN:LR}, where it was applied to realize a moduli stack of Langlands parameters for representations of real reductive groups.  Informally, we view $X/G$ as a family of formal loop spaces over $G/G$, which we in turn view as a family of unipotent loop spaces over $G//G$, ultimately realizing $X/G$ as a family of unipotent loop spaces over the affine parameter space $G//G$.

This paper contains two main theorems.  The first realizes well-established equivariant localization patterns (e.g. as in \cite{GKM:KD} \cite{Th:AK} \cite{Th:LRR}) in the setting of Hochschild homology via its geometric avatar, the derived loop space.  The second realizes an Atiyah-Segal style (e.g. as in \cite{AS:KC} \cite{Th:AS}) completion theorem identifying completed periodic cyclic homology with 2-periodic equivariant derived de Rham cohomology.  Over $\C$ a theorem of Bhatt \cite{Bh:DDR} gives an identification with the Betti cohomology of the analytification.

We mention a few related results in the literature.  An analogous theory was explored by Block and Getzler in \cite{BG:ECH} in the setting of a compact group $G$ acting on a compact smooth manifold $M$.  Similar results appear in the algebraic setting when $G$ is finite (i.e. $X/G$ is a Deligne-Mumford stack) in Theorem 1.15 of \cite{ACH:DI}, and in the case of an smooth affine quotient stack in Lemma 4.11 and Proposition 4.12 of \cite{HLP:EH}.  We generalize their statements to the case of a general smooth quasi-projective quotient stack.  Our statements have also been investigated in the setting of smooth quotient stacks with finitely many orbits in \cite{BZN:LR}, with special attention to the case $B \bs G/B$ in Theorem 3.5 of \opcit

We will now begin stating our results precisely.  The following theorem is an abridged version of our geometric equivariant localization theorem for derived loop spaces, which appears in the main text as Theorem \ref{MainThm}.
\begin{theorem}[Equivariant localization for derived loop spaces]\label{MainThmIntro}
Let $G$ be a reductive group acting on a smooth variety $X$ over an algebraically closed field $k$ of characteristic 0, and $z \in G$ a semisimple element.  Note that $\LL(X/G)$ naturally lives over $\LL(BG)$; let $\LLf_z(X/G)$ denote the completion of $\LL(X/G)$ along the fiber over the semisimple orbit $\{G \cdot z\}/G \subset \LL(BG)$ and let $\LL_z^u(X/G)$ denote the completion over the saturation $[z] \in G//G$.  For $X$ a smooth variety with a $G$-action, there are functorial $S^1$-equivariant isomorphisms
$$\begin{tikzcd}\LLf_z(\pi_0(X^z)/G^z) \arrow[r, "\simeq"] & \LLf_z(X/G), & & \LL^u_z(\pi_0(X^z)/G^z) \arrow[r, "\simeq"] & \LL^u_z(X/G).\end{tikzcd}$$
\end{theorem}
We remark briefly on the assumptions of the theorem.  We assume that $k$ is algebraically closed and that $z$ is semisimple in reductive $G$ since we argue via the Luna slice theorem, and we work in characteristic 0 since we model derived schemes locally as dg algebras.  The assumption that $X$ is smooth is used to compute via Koszul resolutions, and is essential since the statement is false otherwise.

We note that in the case that $G$ is a torus, these statements can be strengthened: the above maps are equivalences on a Zariski open neighborhood over $z$, recovering the equivariant localization in \cite{CG:RTCG} for $K$-theory in the setting of periodic cyclic homology.  However, for a nonabelian reductive group $G$, this fails even in the case when $X$ is a point (see Remark \ref{RmkTorus} and Example \ref{LocalFailure}).  Our result implies the following interpretation of derived fixed points, which also appears in \cite{ACH:DI} as Corollary 1.12.
\begin{cor}\label{DerivedInv}
Let $G$ be a reductive group acting on a smooth variety $X$ and $z \in G$ semisimple.  We have a natural identification of the derived $z$-fixed points:
$$\begin{tikzcd}
\LL(\pi_0(X^z)) \arrow[r, "\simeq"] & X^z := \LL(X/G) \times_{BG} \{z\}.\end{tikzcd}$$ 
\end{cor}
\begin{proof}
The ``shift by $z$'' map on $\LL(\pi_0(X^z)/G^z))$ is a (non-$S^1$-equivariant!) equivalence, so that 
$$\LLf_z(\pi_0(X^z)/G)) \simeq \LLf(\pi_0(X^z)/G^z),$$
and in particular since loop spaces commuted with fiber products, $X^z \simeq \LL(\pi_0(X^z))$.  
\end{proof}

The unipotent version of equivariant localization for loop spaces implies an equivariant localization result for Hochschild homology and its cyclic variants.  The following appears in the main text as Theorems \ref{ThmHH} and \ref{ThmHC}.
\begin{cor}[Equivariant localization for Hochschild homology]
Let $G$ be reductive group acting on a smooth variety $X$, and $z \in G$ a semisimple element.  Then we have an $S^1$-equivariant equivalence
$$\begin{tikzcd} HH(\Perf(X/G))_{\widehat{z}} \arrow[r, "\simeq"] &  HH(\Perf(\pi_0(X^z)/G^z))_{\widehat{z}}\end{tikzcd}$$
and similarly when replacing $HH$ with its cyclic variants $HC, HN,$ and $HP$.
\end{cor}
Note that since the formation of periodic cyclic homology involves a colimit, this is not automatic in that case.  The fact that $X$ is smooth gives us a cohomological boundedness of Hochschild homology, which is essential in establishing the above result.

After identifying the completed derived loop spaces over a central character $z$, we are interested in identifying this completion with de Rham cohomology via an analogue of the Atiyah-Segal completion theorem in the setting of periodic cyclic homology.  The following theorem is a consequence of Proposition \ref{StackCoh} and Theorem \ref{FormalUni}.
\begin{theorem}[Atiyah-Segal completion for periodic cyclic homology]
Let $X/G$ be a global quotient stack where $X$ is an algebraic space.  Then, there is an equivalence
$$\begin{tikzcd} HP(\Perf(X/G))_{\widehat{e}} \arrow[r, "\simeq"] & C^\bullet_{dR}(X/G; k) \,\widehat{\otimes}^!_k \,k((u))\end{tikzcd}$$
where $C^\bullet_{dR}(X/G; k)$ denotes the derived de Rham cohomology, and $\widehat{\otimes}^!$ indicates completion with respect to the coarsest topology induced by the derived Hodge filtration and $u$-adic filtration on respective tensor factors.
\end{theorem}

When $k=\C$, applying a generalization of the main theorem of \cite{Bh:DDR} to geometric stacks, we can identify Tate functions on formal loop spaces with Betti cohomology. 
\begin{cor}
Let $X/G$ be a finite type global quotient stack over $k=\C$, where $X$ is an algebraic space.  Then, there is an equivalence
$$\begin{tikzcd} HP(\Perf(X/G))_{\widehat{e}} \arrow[r, "\simeq"] & C^\bullet_{dR}(X^{an}/G^{an}; \C) \,\widehat{\otimes}^! \,\C((u))\end{tikzcd}$$
where $C^\bullet_{dR}(X^{an}/G^{an}; \C)$ denotes de Rham cohomology of the analytification and $\widehat{\otimes}^!$ indicates completion with respect to the coarsest topology induced by the Hodge filtration in the Cartan model for equivariant cohomology and $u$-adic filtration on respective tensor factors.
\end{cor}

The main technical hurdle in the proof of the theorem is that the fiber of $\LL(X/G)$ over $[e] \in G//G$ does not just contain formal loops but also unipotent loops.   In Theorem \ref{FormalUni} we show that this difference vanishes after applying the Tate construction.
\begin{theorem}\label{IntroUniFormal}
For $X$ a quasicompact algebraic space acted on by an affine algebraic group $G$, the pullback functor on derived global functions induces an isomorphism
$$\begin{tikzcd} \OO(\LL^u(X/G))^{\Tate} \arrow[r, "\simeq"]& \OO(\LLf(X/G))^{\Tate}.\end{tikzcd}$$
In particular, if $U$ is a unipotent group, then the pullback functor induces an equivalence
$$\begin{tikzcd}HP(\Perf(X/U)) \arrow[r, "\simeq"] & HP(\Perf(X)).\end{tikzcd}$$
\end{theorem}

Following an analysis of the twisted $S^1$-rotation action on $\LL(X/G)$ above points of $G/G$ away from the identity in Section \ref{TwistedActionSection}, we relate this completion theorem to our localization theorem in Theorem \ref{HPLoc}.
\begin{theorem}\label{IntroHPLoc}
Let $G$ be a reductive group acting on a smooth quasi-projective variety $X$.  The periodic cyclic homology $\HP(\Perf(X/G))$ is naturally a module over 
$HP(\Perf(BG)) = k[G//G]((u))$.  For a closed point $z \in G//G$, we have an identification of the formal completion at $z$ with 2-periodic Betti cohomology of the $z$-fixed points
$$\begin{tikzcd}HP(\Perf(X/G))_{\widehat{z}} \arrow[r, "\simeq"] & C^\bullet_{dR}(\pi_0(X^z)/G^z; k) \, \widehat{\otimes}_k^!\,  k((u))\end{tikzcd}$$
as a module over $HP(\Perf(BG))_{\widehat{z}} \simeq C_{dR}^\bullet(BG^z; k) \, \widehat{\otimes}_k^!\,  k((u))$, contravariantly functorial with respect to $X$.
\end{theorem}
Using Corollary \ref{DerivedInv}, we also obtain an identification of the (derived) specialization of periodic cyclic homology at $z \in G$ with non-equivariant cohomology of $z$-fixed points.
\begin{cor}
Let $X$ be a smooth variety with an action of a reductive group $G$.  For $z \in G$ semisimple, let $k_z$ denote the skyscraper sheaf at $[z] \in G//G$.  We have an equivalence
$$\begin{tikzcd} HP(\Perf(X/G)) \otimes_{k[G/G]((u))}^L k_z((u)) \arrow[r, "\simeq"] & C_{dR}^\bullet(\pi_0(X^z); k)((u)).\end{tikzcd}$$
\end{cor}



We indicate two natural directions in which our results may be extended.  The first is to ask what happens when $X$ is allowed to be singular; in this case, one can study either the Hochschild homology of $\Coh(X)$ or $\Perf(X)$.  In the case when $X$ is a fiber product of smooth schemes, the geometric statement follows immediately from our results, but it is unclear to us how to generalize beyond that case.  If a geometric statement is out of reach, it is also of interest as to whether the global localization statements for periodic cyclic homology hold; an obstruction to applying the standard techniques of embedding a singular quotient stack into a smooth one is the lack of a devissage theorem for the periodic cyclic homology of stacks.  A second direction would be to categorify these results in a generalization of the Koszul duality of \cite{Pr:IL} and \cite{BZN:LC}.  In addition, it would be pleasing to have a more conceptual proof of Theorem \ref{FormalUni}.

We end our introduction with a few toy examples.
\begin{exmp}
Let $G = \G_m = \Spec k[z,z^{-1}]$ act on $X = \mathbb{A}^1 = \Spec k[x]$ by scaling, i.e. assign $\G_m$-weight $|x| = 1$.  The loop space can be calculated directly via Proposition \ref{LoopQuotient}
$$\LL(\mathbb{A}^1/\G_m) = \frac{\Spec k[z,z^{-1},x]/\langle x(z-1)\rangle}{\G_m}.$$
The Hochschild homology and periodic cyclic homology can also be calculated directly
$$HH(\Perf(\mathbb{A}^1/\G_m)) = (k[z,z^{-1},x]/x(z-1))^{\G_m} = k[z,z^{-1}], \;\;\;\;\;\; HP(\Perf(\mathbb{A}^1/\G_m)) = k[z,z^{-1}]((u))$$
as the $S^1$-equivariant structure on a complex concentrated in a single cohomological degree can only be realized by the zero map.  Completing at any $z_0 \in \G_m$ gives, for $t = z - z_0$ and $|t| = 0$,
$$HP(\Perf(\mathbb{A}^1/\G_m))_{\widehat{z}} \simeq k[[t]]((u))$$
where $t = z - z_0$.

On the other hand, we can compute $H^\bullet((X^z)^{an}; k)((u))$ for each $z_0$.  For $z = z_0 \ne 1$, the fixed points $\pi_0(X^{z_0})/G^{z_0} = \{0\}/\G_m \simeq BS^1$, whose 2-periodic cohomology is $H^\bullet(BS^1; k) \otimes^! k((u)) = k[[s]] \otimes^! k((u)) $ with $|s| = 2$.  For $z = 1$ the fixed points are $\mathbb{A}^1/\G_m \simeq \mathbb{C}/S^1 \simeq BS^1$, and the same argument applies.  The identification $k[[s]]((u)) \simeq k[[t]]((u))$ is by $tu = s$; in particular it is necessary to invert the degree 2 operator $u$.  The discrepancy between the cohomological degrees of $t$ and $s$ is a manifestation of the Koszul duality degree-weight shearing discussed in \cite{BZN:LC}.
\end{exmp}

\begin{exmp}[Flag variety]\label{ExmpFlag}
Let $X = G/B$ be the flag variety with the usual action of $G$.  Then, $X/G = BB$, so $\LL(X/G) = B/B = \widetilde{G}/G$ is the Grothendieck-Springer resolution; the fiber for the map $\widetilde{G} \rightarrow G$ over any point $g \in G$ consists of the Borel subgroups containing $g$, i.e. the $g$-fixed points of $G/B$.  
We identify the Hochschild homology $HH(\Perf(BB))$ as a $HH(\Perf(BG))$-module by the inclusion map
$$HH(\Perf(BB)) = \OO(\tilde{G}/G) = k[H] \leftarrow HH(\Perf(BG)) = \OO(G/G) = k[H]^W$$
where $H$ is the universal Cartan subgroup and $W$ is the universal Weyl group\footnote{It is known that $\OO(\tilde{G})$ has vanishing higher cohomology.}.  Let $s \in G$ be a semisimple element, and $[s]$ its adjoint orbit.  Completing at $[s] \in k[H]^W$, we have
$$HH(\Perf(BB))_{\widehat{[s]}}  = \bigoplus_{|W \cdot s|} k[[\h]] \leftarrow HP(\Perf(BG))_{\widehat{s}} \simeq k[[\h]]^{W_{G^s}}.$$
In particular, the rank of $HP(\Perf(BB))_{\widehat{s}}$ over $HP(\Perf(BG))_{\widehat{s}}$ is $|W \cdot s| \cdot |W_{G^s}| = |W|$ by a theorem of Steinberg and Pittie \cite{St}.  
Note that $\h$ is placed in cohomological degree zero.  Applying the Tate construction, we find that
$$HP(\Perf(BB))_{\widehat{[s]}}   \simeq k[[\h]]^{W_{G^s}}((u)).$$

On the other hand, the fixed points $(G/B)^s$ consist of Borel subgroups containing $s$; by conjugating, we can choose a torus such that $s \in T \subset B$; let $\mf{t}$ be the Lie algebra of $T$.  There is a $G^s$-action on $(G/B)^s$ and its stabilizer at every point is conjugate to $B^s$, but the action may not be transitive; thus, $(G/B)^s$ is the disjoint union of copies of $G^s/B^s$.  To count the number of connected components, we count $T$-fixed points: the $T$-fixed points of $G/B$ are also $s$-fixed points, and furthermore each $G^s/B^s$ contains $|W_{G^s}|$ such $T$-fixed points, so we have $|W|/|W_{G^s}|$ connected components.  Finally, accounting for $T$-equivariance, we have
$$H^\bullet(\coprod_{|W|/|W \cdot s|} B(B^s)^{an}; k) = \bigoplus_{|W|/|W^{G^s}|} k[[\h]] \leftarrow H^\bullet(B(G^s)^{an}; k) = k[[\mf{t}]]^{W_{G^s}}$$
where $\mf{t}$ is placed in cohomological degree 2.  In particular, 
$$HP(\Perf(BB))_{\widehat{[s]}} \simeq k[[\h]]^{W_{G^s}}((u)) \simeq k[[\mf{t}]]^{W_{G^s}}((u)) \simeq H^\bullet((BB^s)^{an}; k) \otimes^! k((u))$$
under the isomorphism $u\h \simeq \mf{t}$.
\end{exmp}

\subsection{Conventions and notation}

In this note, $k$ will denote an algebraically closed field of characteristic zero, and we work over $\pt = \Spec(k)$.  Unless otherwise stated, all functors and categories are derived, e.g. for an affine scheme $X = \Spec(A)$, we denote by $\QCoh(X)$ the derived category of unbounded complexes of $A$-modules localized with respect to quasi-isomorphisms, and $\otimes = \otimes^L$ although we sometimes use the latter notation for emphasis (e.g. when performing calculations).  We indicate a functor that is not derived by writing $\pi_0$ or $H^0$.

All gradings follow cohomological grading conventions (i.e. differentials increase degree), unless otherwise indicated by a subscript, and $HH$ will always denote the cochain complex of Hochschild chains rather than its cohomology groups, which we denote $H^\bullet(HH)$.  We refer to the $n$th cohomology group of a chain complex $V$ by $H^n(V) = \pi_{-n}(V)$.

We require a theory of $\infty$-categories and derived algebraic geometry.  Such theory has been developed by by To\"{e}n and Vezzosi in \cite{TV:HAG1} \cite{TV:HAG2} and by Lurie in \cite{Lu:HA} \cite{Lu:HTT} \cite{Lu:SAG} \cite{Lu:DAGIV} \cite{Lu:DAGVII} \cite{Lu:DAGVIII} \cite{Lu:DAGXII}.  Below, we summarize some of the main definitions.
\begin{rmk}[$\infty$-categories]
By \emph{$\infty$-category} we mean an $(\infty, 1)$-category, and we do not specify a particular model\footnote{A forthcoming book by Riehl and Verity \cite{RV:E} establishes the model-independence of $(\infty, 1)$-categories and its foundational properties, constructions, and theorems.}.  We let $\cat{S}$ denote the $\infty$-category of \emph{$\infty$-groupoids} or \emph{spaces} and we will take for granted that the category of $\infty$-categories is enriched in $\cat{S}$.  We let $\cat{st}_k$ denote the $\infty$-category of small (stable) $k$-linear $\infty$-categories whose 1-morphisms are $k$-linear exact functors, and $\cat{Pr}_k^L$ the category of presentable (stable) $k$-linear $\infty$-categories whose 1-morphisms are functors which are $k$-linear exact left adjoints\footnote{In particular, by Remark 6.5 in \cite{Lu:DAGVII}, $k$-linear presentable categories are automatically stable.  By the adjoint functor theorem, left adjoint functors commute with filtered colimits.}.  Note that presentable $\infty$-categories admit a combinatorial model structure.

For such a category $\cat{C} \in \cat{Pr}_k^L$, we let $\cat{C}^{\omega} \in \cat{st}_k$ denote its compact objects.  For $\cat{C} \in \cat{st}_k$, we let $\Ind(\cat{C}) \in \cat{Pr}^L_k$ denote its ind-completion.  By \cite{Co:DG}, a presentable $k$-linear $\infty$-category in $\cat{Pr}^L_k$ has an associated $k$-linear differential graded category in $\cat{dgcat}_k$.  We will denote by $\Fun^L_k(-, -)$ and $\Fun^R_k(-,-)$ the spaces of $k$-linear exact functors which are left and right adjoints respectively.  For more details, see Chapter 5 of \cite{Lu:HTT}, Section 1.4.4 in \cite{Lu:HA}, and Section 6 of \cite{Lu:DAGVII}.
\end{rmk}

\begin{rmk}[Derived stacks]
We let $\cat{DRng}$ denote the $\infty$-category of \emph{derived rings} (or derived algebras over $k$); during our exposition we do not insist on a particular model, but we will always compute in the category of dg algebras over $k$ with its projective model structure.  The opposite category $\cat{Aff} = \cat{DRng}^{op}$ is defined to be the category of \emph{affine derived schemes}.  A \emph{derived scheme} is as a derived locally ringed space whose 0-truncation is a scheme and whose higher homotopy groups are quasicoherent \cite{To:DAG} \cite{Lu:SAG}.  The global sections functor and derived spectrum functors induce equivalences identifying $\cat{Aff}$ with the category of derived schemes whose $\pi_0$ is affine in the classical sense.  We will refer to derived schemes as simply \emph{schemes}, and use the term \emph{classical scheme} to refer to a derived scheme $X$ for which $\pi_0(X) = X$.  

A \emph{dg scheme} \cite{CFK:DQ} or \emph{embeddable derived scheme} is defined somewhat differently; it is defined to be a scheme $(X, \OO_X)$ along with a non-positively graded sheaf of complexes $\OO_X^\bullet$ such that $\OO_X^0 = \OO_X$ and $H^n(\OO_X^\bullet)$ are quasicoherent.  In particular, a dg scheme $Z = (X, \OO_X^\bullet)$ admits an embedding $Z \rightarrow X$ into the classical scheme $X = (X, \OO_X)$.  Every derived scheme is locally modeled by a dg scheme.

A \emph{prestack} is an $\infty$-functor $\cat{Aff}^{op} := \cat{DRng} \rightarrow \cat{S}$, and a (derived) \emph{stack} is a prestack which is a sheaf for the derived \'{e}tale topology \cite{GR:DAG} \cite{TV:HAG2}.  We mean (derived) \emph{algebraic stack} in the sense of \cite{DG:QCA}: an (derived) Artin 1-stack whose diagonal is quasi-separated, quasi-compact, and representable by (derived) schemes and admits an atlas by a (derived) scheme.  We mean \emph{geometric stack} in the sense of \cite{BZN:LC}: an algebraic stack whose diagonal map is affine.  We say an algebraic stack is \emph{quasi-compact} if it admits a quasi-compact atlas $U$ (equivalently, if it admits an affine atlas).  A map of prestacks $X \rightarrow Y$ is \emph{schematic} if for any scheme $S$ and map $S \rightarrow Y$, the base change $X \times_Y S$ to $S$ is a scheme.  A map of derived schemes $f: X \rightarrow Y$ is a \emph{closed immersion} if it is in the classical sense on $\pi_0$; a map of algebraic stacks is a \emph{closed immersion} if it is after base change to an atlas.
\end{rmk}

\subsection{Acknowledgements}

I would like to thank David Nadler for our numerous conversations regarding this topic and for suggestions leading to and during the writing of this paper.  I would also like to thank Daniel Halpern-Leistner for suggestions regarding the equivariant cyclic bar complex, David Ben-Zvi for suggestions regarding twisted circle actions and Bhargav Bhatt for explaining how his results in \cite{Bh:DDR} extend to the case of stacks.  Finally, I would like to thank Shishir Agrawal for numerous helpful discussions, and Brian Hwang for his comments on this paper.  This work was mostly carried out at UC Berkeley, in part at Cornell University, and was partially supported by NSF RTG grant DMS-1646385.

\section{Background}\label{BGSection}

In this section we provide some basic exposition on Hochschild homology, loop spaces and derived algebraic geometry.  At parts it is an informal summary of the existing literature, and at parts we provide proofs of some folklore likely known to experts.

\subsection{Derived loop spaces and its variants}\label{LoopSpaceSection}

An in-depth discussion of derived loop spaces, which we often simply refer to as loop spaces, can be found in \cite{BZN:LC}.  We will summarize the main definitions and prove some foundational results in the case of the loop space of an algebraic or geometric stack.  These should probably be skipped on a first reading; the essential statements for the main body of the paper are in Propositions \ref{FormalLoopsQuotient} and \ref{UnipotentLoopsQuotient}, which provide an explicit description of the formal and unipotent loops of a global quotient stack.

\begin{defn}
We consider the higher derived stack $S^1$ as the locally constant sheaf on $\cat{Aff}$ with value the topological circle $S^1$.  Its affinization is the shifted affine line $B\G_a = \Spec C^\bullet(S^1; k)$ and the map $S^1 = B\mathbb{Z} \rightarrow B\G_a$ is induced by the map of abelian groups $\mathbb{Z} \rightarrow \G_a$.   \end{defn}

\begin{rmk}
The stack $B\G_a$ is not an affine scheme since $C^\bullet(BS^1; k)$ is not connective, but it still has a well-defined functor of points; it is an example of a \emph{coaffine stack} (in the language of \cite{Lu:DAGVIII}) or an \emph{affine stack} (in the language of \cite{To:AS}).  Explicitly, by Lemma 2.2.5 in \cite{To:AS} or the introduction to Section 4 of \cite{Lu:DAGVIII}, it is the right Kan extension\footnote{That is, the (fully faithful) left adjoint to the restriction of a prestack (i.e. a functor $\cat{DRng} \rightarrow \cat{S}$) to a classical prestack (i.e. a functor $\cat{Rng} \rightarrow \cat{S}$).} of the classical stack\footnote{In fact, coaffine stacks are always left Kan extensions of classical stacks.} sending an affine scheme $S = \Spec(R)$ to the Eilenberg-Maclane space $K(R, 1)$ where $R$ is considered as an abelian group under addition.  The affinization map $S^1 \rightarrow B\G_a$ is given on $S$-points by the map of Eilenberg-Maclane spaces $K(1, \mathbb{Z}) \rightarrow K(1, S)$ where we consider $S$ as an abelian group under addition.  We fix an isomorphism $C^\bullet(S^1; k) \simeq k[\eta]$ where $|\eta| = 1$.

\end{rmk}

\begin{defn}\label{DefnLoopSpace}
Let $X$ be a prestack.  We define the derived loop space and its variants as follows.
\begin{itemize}
\item The \emph{(derived) loop space} of $X$ is the derived mapping stack 
$$\LL(X) := \Map(S^1, X) \simeq X \times_{X \times X} X.$$ 
The second presentation is a consequence of the presentation of $S^1$ via the homotopy pushout $S^1 \simeq \pt \coprod_{S^0} \pt = \Sigma S^0$, and the property that derived mapping stacks take coproducts in the source to products.  The evaluation map $p: \LL X \rightarrow X $ realizes the loop space as a relative group stack over $X$.  The derived loop space has a canonical $S^1$-action by loop rotation.
\item The \emph{formal loop space} $\LLf(X)$ is the completion of $\LL(X)$ along constant loops $X \rightarrow \LL X$.  It inherits a loop rotation $S^1$-action from $\LL X$.
\item The \emph{unipotent loop space} $\LL^u(X)$ is the derived mapping stack
$$\LL^u(X) := \Map(B\G_a, X)$$
and the affinization map $S^1 \rightarrow B\G_a$ defines a map $\LL^u X \rightarrow \LL X$.  There is a natural $B\G_a \rtimes \G_m$-action on $\LL^u X$ arising from the natural $\G_m$-action on $B\G_a$, compatible with the $S^1$-action on $\LL X$.
\item If $X$ admits a cotangent complex, we define the \emph{odd tangent bundle}, a linearized form of the loop space, by
$$\mathbb{T}_X[-1] := \Spec_X \Sym_X \mathbb{L}_X[1]$$
i.e. the relative spectrum of the derived symmetric powers\footnote{We define the relative spectrum as follows: for an algebra object $\mathcal{A} \in \QCoh(X)$, we define the $S$-points for $\Spec_X \mathcal{A}$ as pairs $(\eta, \delta)$ where $\eta \in X(S)$ and $\delta: S \rightarrow \Spec \eta^*\mathcal{A}$ which are compatible under the projection; note that $\eta^*\mathcal{A}$ is an algebra since $S$ is an affine derived scheme and pullback preserves the monoidal structure on quasicoherent sheaves.  The symmetric algebra functor $\Sym_X$ is left adjoint to the forgetful functor from the category of augmented commutative unital associative algebra objects of $\QCoh(X)$, which exists by the adjoint functor theorem.} of the cotangent complex.  There is a projection $q: \mathbb{T}_X[-1] \rightarrow X$ and a zero section $c: X \rightarrow \mathbb{T}_X[-1]$ induced by the structure and augmentation maps respectively.  We write $\widehat{\mathbb{T}}_X[-1]$ for the odd tangent bundle completed at its zero section.  Both $\mathbb{T}_X[-1]$ and $\widehat{\mathbb{T}}_X[-1]$ are equipped with a natural $B\G_a \rtimes \G_m$-action, where the $B\G_a$-action is encoded by the de Rham differential and the $\G_m$-action is by scaling on the fibers.
\end{itemize} 
\end{defn}

\begin{defn}
	Let $X$ be a quasicompact geometric stack.  There is an \emph{exponential map} 
	$$\exp: \widehat{\mathbb{T}}_X[-1] \rightarrow \LLf X$$
	defined in Section 6 of \cite{BZN:LC}.  In particular, $\LLf X$ has a natural $B\G_a \rtimes \G_m$-action compatible with the $S^1$-action.
\end{defn}

\begin{thm}[Hochschild-Kostant-Rosenberg]
The exponential map is an equivalence.
\end{thm}
\begin{proof}
For the stacky case, see Section 6 of \cite{BZN:LC}; when $X$ is a derived scheme, see the main theorem of \cite{BF}.
\end{proof}

\begin{exmp}
If $X = \Spec(A)$, then the derived loop space 
$$\LL(X) = \Spec(A \otimes_{A \otimes_k A^{op}} A) = \Spec(C^\bullet(A; A))$$
is the derived spectrum of the cyclic bar complex equipped with the shuffle product.  The rotation $S^1$-action has a combinatorial realization via the cyclic structure on the cyclic bar complex \cite{Lo:CH} \cite{Jo:CH}.  In this example, we think of the bar resolution $B^\bullet(A) \rightarrow A$ as the $A \otimes A^{op}$-module obtained by tensoring $A$ with the map of simplicial complexes $I \rightarrow \pt$, where the unit interval $I$ is presented by a simplicial set with two 0-simplices and one non-degenerate 1-simplex:
$$B^\bullet(A) = A \otimes I \rightarrow A \;\;\;\;\;\;\;\;\;\; X \simeq \Map_{DSt}(I, X) = \Spec(A \otimes I).$$
The cyclic bar complex $C^\bullet(A) = B^\bullet(A) \otimes_{A \otimes A^{op}} A$ is obtained by gluing the two 0-simplices of $I$, i.e. it is the chain complex associated to the tensor product of $A$ with the presentation of $S^1$ by one 0-simplex and one non-degenerate 1-simplex:
$$C^\bullet(A) = A \otimes S^1 \;\;\;\;\;\;\;\;\;\; \LL(X) = \Map_{DSt}(S^1, \Spec(A)) = \Spec(A \otimes S^1).$$ 
This makes $\LL X$ into a cocyclic scheme and $\OO(\LL X)$ into a cyclic algebra.
\end{exmp}

\begin{exmp}
If $X$ is a stack, then $\pi_0(\LL X)$ is the (classical) inertia stack of $X$, so $\LL X$ can be thought of as a derived inertia stack.  In particular, let $X = BG$; then $\LL(BG) = G/G$ is the stacky adjoint quotient (see Proposition \ref{LoopQuotient} below).  Note that $\LL(BG) = BG \times_{BG \times BG} BG$ is classical since the diagonal map is flat.  The $S^1$-equivariant structure on $\OO(G/G)$ has a description in terms of the a cyclic algebra arising from the cyclic structure on the simplicial Cech nerve for the atlas $G \rightarrow G/G$ (see Section 7.3.3 of \cite{Lo:CH}).
\end{exmp}

\begin{prop}[Loop space of a quotient stack]\label{LoopQuotient}
	The loop space of a quotient stack $\LL(X/G)$ can be computed by the $G$-equivariant fiber product
	$$\begin{tikzcd}
	\LL(X/G) \arrow[r] \arrow[d] & (X \times G)/G \arrow[d, "a \times p"] \\
	X/G \arrow[r, "\Delta"] & (X \times X)/G
	\end{tikzcd}$$
	where $G$ acts on $X \times X$ and $X \times G$ diagonally.
\end{prop}
\begin{proof}
	Note that $X/G \times X/G \simeq (X \times X)/(G \times G)$ with action $(g_1, g_2) \cdot (x_1, x_2) = (g_1 x_1, g_2 x_2)$.  We write 
	$$\frac{X}{G} \times_{\frac{X\times X}{G \times G}} \frac{X}{G} = \frac{X \times G}{G \times G} \times_{\frac{X \times X}{G \times G}} \frac{X \times G}{G \times G}$$
	where the map $X \times G \rightarrow X \times X$ sends $(x, g) \mapsto (x, gx)$ and the action of $G \times G$ on $X \times G$ is $(g_1, g_2) \cdot (x, g) = (g_1 x, g_2gg_1^{-1})$.  
	The claim follows from the ``two-out-of-three'' lemma for pullback squares applied to
	$$\begin{tikzcd}
	\LL(X/G) \arrow[r] \arrow[d] & (X \times G)/G \arrow[r] \arrow[d] & (X \times G)/(G \times G) \arrow[d] \\
	X/G \arrow[r] & (X \times X)/G \arrow[r] & (X \times X)/(G \times G).
	\end{tikzcd}$$
\end{proof}

\begin{rmk}
	Forgetting $G$-equivariance, the geometric points of the loop space $\LL(X/G) \times_{BG} \pt$ are given by
	$$(\LL(X/G) \times_{BG} \pt)(k) = \{(x, g) \in X(k) \times G(k) \mid g \cdot x = x\}.$$
	The geometric fiber of the map $\LL(X/G) \rightarrow \LL(BG) = G/G$ over $g \in G(k)$ is the fixed points $X^g(k)$.  The geometric fiber of the evaluation map $\LL(X/G) \rightarrow X/G$ over $x \in X(k)$ is the stabilizer of $x$ in $G(k)$.
\end{rmk}

\begin{exmp}[Odd tangent bundle of smooth quotient stacks]\label{OddTangentBundle}
	In the case of $X/G$ where $X$ is smooth, we have that 
	$$\mathbb{L}_{X/G} = \left( \Omega^1_X \rightarrow \g^* \otimes \OO_X \right)$$
	where the internal differential $d$ is the Cartan differential:
	$$ \Sym^n(\mathbb{L}_{X/G}[1]) = \Sym^n( \Omega^1_X \rightarrow \g^* \otimes \OO_X) \simeq (\Omega^{n}_X \rightarrow \g^* \otimes \Omega^{n-1}_X \rightarrow \cdots \rightarrow  \Sym^n(\g^*) \otimes \OO_X),$$
	$$p_* \OO_{\widehat{\mathbb{T}}_{X/G}[-1]} = \left( \lim_k \bigoplus_{i \geq k} \frac{\Sym^\bullet \g^*}{\Sym^{\geq k} \g^*}  \otimes \Omega^{i}_X[i], d\right).$$
	The resulting de Rham complex is called the \emph{Cartan model} for equivariant cohomology; see Proposition 4.12 of \cite{HLP:EH} for more discussion.  This example can also be carried out when $X$ is not smooth, replacing $\Omega^1_X$ with $\mathbb{L}_X$.
\end{exmp}

\subsubsection{Loop spaces of algebraic and geometric stacks}

We prove some technical facts which may be skipped on a first reading.  Note that a quasi-compact geometric stack is automatically QCA in the sense of \cite{DG:QCA}.

\begin{rmk}
The following principles are standard and will be used frequently.  If $X$ is an algebraic stack then $X$ admits a cover by a disjoint union of affine schemes; if $X$ is quasi-compact this disjoint union can be taken to be finite, so that $X$ admits a cover by an affine scheme.  If $X$ is geometric (i.e. has affine diagonal), then 
$$S \times_X T = (S \times T) \times_{X \times X} X$$
is affine for any affine schemes $S, T$.
\end{rmk}

\begin{lemma}\label{LoopsGeometric}
	Let $X$ be an algebraic stack.  Then $\LL X$ is an algebraic stack.  If $X$ is geometric, then $\LL X$ is geometric. If $X$ is geometric and quasi-compact, then so is $\LL X$.
\end{lemma}
\begin{proof}
	Assume $X$ is algebraic.  That $\LL X$ is algebraic follows from the fact that $\LL X = \Map(S^1, X)$ is a finite limit, and any finite limit of algebraic stacks is algebraic.  An algebraic stack $X$ is geometric if and only if for any map from an affine $U \rightarrow X$, the stack $U \times_X U$ is an affine scheme.		
 In particular, $U \times_{X \times X} X$ is a cover for $\LL X$, and we have
	$$(U \times_{X \times X} X) \times_{\LL X} (U \times_{X \times X} X) = U \times_X (\LL X \times_{\LL X} (U \times_{X \times X} X)) = (U \times_X U) \times_{X \times X} X$$
	which is affine since $U \times_X U$ is affine and the diagonal map is affine, so $\LL X$ is geometric.  Assume $X$ is also quasi-compact; then it admits a cover by an affine $U$, and $U \times_{X \times X} X$ is also affine since the diagonal is affine.
\end{proof}

\begin{lemma}\label{FormalLoopsSchematic}
	Let $X$ be an algebraic stack.  Then the inclusion of constant loops $X \rightarrow \LL(X)$ is a (schematic) closed immersion.
\end{lemma}
\begin{proof}
	Since the diagonal map $X \rightarrow X \times X$ is representable by schemes, so is the evaluation map $\LL(X) \rightarrow X$.  Let $U \rightarrow X$ be an atlas for $X$ with $U$ a scheme; its base change along the evaluation map gives a cover by a scheme $U \times_{X \times X} X \rightarrow \LL X$.  In particular, by the two-out-of-three property of Cartesian squares, the left square is Cartesian
	$$\begin{tikzcd}
	U \arrow[r] \arrow[d] & U \times_{X \times X} X \arrow[d] \arrow[r] & U \arrow[d] \\
	X \arrow[r] & \LL X \arrow[r] & X
	\end{tikzcd}$$
	i.e. the base change of the inclusion of constant loops $X \rightarrow \LL X$ along an atlas is a scheme, so it is schematic.  It is a closed embedding since any map of derived schemes which admits a retract is a closed embedding, and $U \rightarrow U \times_{X \times X} X$ admits a retract by universal property.
	
	We provide a proof for the last claim.  It suffices to assume all schemes are classical, since the the property of being a closed immersion depends only on classical schemes and the property of admitting a retract is preserved by $\pi_0$.  Let $f: Z \rightarrow Y$ be a map of schemes admitting a retract.  We can verify that $f$ is a closed immersion affine locally on $Y$, so assume $Y$ is affine.  It is a closed immersion if $f^\sharp: \OO_Y \rightarrow f_* \OO_Z$ is surjective.  Since $Y$ is affine, this is equivalent to $\OO(Y) \rightarrow \OO(Z)$ being surjective, which follows since the composition on global functions $\OO(Z) \rightarrow \OO(Y) \rightarrow \OO(Z)$ is the identity.
\end{proof}

We now introduce the notion of based loops of a stack.  Namely, given a point of a stack, the based loop space is the group of automorphisms of that point and the unipotent based loops consist of the unipotent automorphisms.  We use these characterizations in Propositions \ref{FormalLoopsQuotient} and \ref{UnipotentLoopsQuotient} to give explicit descriptions of the formal and unipotent loop spaces of quotient stacks.

\begin{defn}
	Let $X$ be a prestack, and $x: S \rightarrow X$ be an $S$-point where $S$ is an affine derived scheme.  The \emph{group of based loops} at $x$, which we denote $\Omega(X, x)$, is the $\infty$-group object\footnote{The notion of an $\infty$-group can be found in Defintion 7.2.2.1 of \cite{Lu:HTT}, a useful characterization in Proposition 7.2.2.4 of \opcit.  Proposition 6.1.2.11 of \opcit shows that Cech nerves are $\infty$-groupoid objects in any $\infty$-category, so that $\Omega(X, x)$ is a groupoid object in derived stacks and therefore a group object in derived stacks over $S$.  We use the notation $\Hom_{grp}$ to denote the space of group homomorphisms.} in prestacks over $S$ defined to be the Cech nerve of the map $x: S \rightarrow X$, i.e. its underlying derived stack is $\LL X \times_X S$ or equivalently the pullback
	$$\begin{tikzcd}
	\Omega(X, x) \arrow[r] \arrow[d] & S \arrow[d, "x"] \\
	S \arrow[r, "x"] & X.
	\end{tikzcd}$$
	If $f: X \rightarrow Y$ is a map of prestacks, with $x \in X(S)$, then there is a natural map of $\infty$-groups $\Omega(f, x): \Omega(X, x) \rightarrow \Omega(Y, f(x))$.  We define the \emph{unipotent based loops} of $X$, denoted $\Omega^u(X, x)$, by the fiber product $\LL^u(X) \times_X S$; there is a natural map $\Omega^u(X, x) \rightarrow \Omega(X, x)$.  Note that the unipotent based loops do not form a group.
\end{defn}

\begin{rmk}
 If $X$ is algebraic, then based loops at $x \in X(S)$ form an $\infty$-group object in derived schemes over $S$, and if $X$ is geometric, the based loops form an $\infty$-group object in affine derived schemes over $S$.  
\end{rmk}

\begin{exmp}
	Let $X$ be an (affine) derived scheme and $x \in X(k)$ a geometric point.  Then, $\Omega(X,x) = \mathbb{T}_{X, x}[-1] = \Spec \Sym x^*\mathbb{L}_{X}[1]$ is the odd tangent space at $x \in X(k)$.  The comultiplication on functions is given by the natural comultiplication on the symmetric algebra and antipode map by the sign morphism.
\end{exmp}


\begin{lemma}\label{UnipotentLoopsDescription}
For any $S$-point $x \in X(S)$, we have a natural identification of prestacks over $S$
$$\Omega^u(X, x) = \Hom_{grp, S}(\G_a \times S, \Omega(X, x))$$
where the natural map $\Omega^u(X, x) \rightarrow \Omega(X, x)$ is identified with evaluation at $1 \in \G_a$.
\end{lemma}
\begin{proof}
Let $f: S' \rightarrow S$ and $\gamma \in \Omega^u(X, x)(S')$ with base point $f^* x$.  Note that 
$$\Omega(X, x) \times_S S' = \Omega(X, f^* x) = \mathrm{Cech}(S' \rightarrow X \times_S S').$$
We will also use $\gamma$ to denote its image in $(\LL^u X)(S')$.  Let $\iota_0: \pt \rightarrow B\G_a$ denote the inclusion of the (additive) identity.  The diagram
$$\begin{tikzcd}
S' \arrow[r, "\mathrm{id}"] \arrow[d, "\iota_0 \times \mathrm{id}_{S'}"] & S' \arrow[d, "x"] \\
B\G_a \times S' \arrow[r, "\gamma"] & X \times_S S'
\end{tikzcd}$$
commutes, and therefore we have a map of the corresponding Cech nerves as $\infty$-groupoids, giving us a map
$$\Omega^u(X, x)(S') \rightarrow \Hom_{grp, S'}(\G_a \times S', \Omega(X, x) \times_S S')$$
for every $S'$, natural in $f$, inducing a map of stacks
$$\Omega^u(X, x) \rightarrow \Hom_{grp, S}(\G_a \times S, \Omega(X, x)).$$

We now produce the inverse to this map.  To do so, we need to produce a map $\Hom_{grp, S'}(\G_a \times S', \Omega(X, f^* x)) \rightarrow \Omega^u(X, x)(S)$ natural in $f: S' \rightarrow S$.  
First, note that $\Omega^u(X, x)(S) = \Map(B\G_a \times S', X) \times_{X(S)} \Map(S', S)$, so we only need to define a map 
$$\Hom_{grp, S'}(\G_a \times S', \Omega(X, f^* x)) \rightarrow \Map(B\G_a \times S', X).$$
Taking geometric realizations of the $\infty$-groups over $S$ considered as simplicial objects in derived stack, we obtain a map $B\G_a \times S' \rightarrow B_S \Omega(X, f^* x)$.  We compose with the map $B_S \Omega(X, f^* x) \rightarrow X$ induced by universal property of geometric realizations applied to the augmentation map $f^*x: S' \rightarrow X$.  We leave the verification that these two maps are inverses to the reader, and naturality with respect to evaluation at $1 \in \G_a$ (essentially since $S^1 = B\Z \rightarrow B\G_a$ is induced by the inclusion of $1 \in \G_a$).
\end{proof}

The following notion of a contracting action will be used in Section \ref{ProGradedSection}.

\begin{defn}\label{DefnContracting}
	Let $X = \Spec(R)$ be an affine scheme with a $\G_m$-action.  We say the $\G_m$-action is \emph{contracting} if it acts by only non-positive weights on $R$.  In this case, the fixed point locus is $Y = \Spec(R^{\G_m})$, and we say the $\G_m$-action contracts to $Y$.  In particular, there are maps $Y \rightarrow X \rightarrow Y$.  More generally, let $X$ be a prestack with a $\G_m$-action, equipped with a $\G_m$-equivariant affine map $p: X \rightarrow Y$ where $Y$ is given the trivial action.  We say the $\G_m$-action \emph{contracts} to $Y$ if for any affine $S$ and map $S \rightarrow Y$, the induced $\G_m$-action on $S \times_Y X$ contracts to $S$.  In particular, this implies there is also a $\G_m$-equivariant section $Y \rightarrow X$.
\end{defn}

\begin{lemma}\label{LemContracting}
	Let $X$ be a quasi-compact geometric stack.  The $\G_m$-actions on $\LL^u X$ and $\mathbb{T}_X[-1]$ contract to the fixed point locus of constant loops.
\end{lemma}
\begin{proof}
	The claim for $\mathbb{T}_X[-1]$ is by definition.  For the unipotent loop space, take $x \in X(S)$.  It suffices to show that the induced $\G_m$-action on $\Omega^u(X, x)$ is contracting.  This follows from the description of $\Omega^u(X, x) = \Hom_{grp,S}(\G_a, \Omega(X, x))$, and the contracting $\G_m$-action on $\G_a$.
\end{proof}

\subsubsection{Formal and unipotent loops over schematic maps}

In \cite{BZN:LC}, it is shown that for $X$ a scheme, $\LLf(X) = \LL^u(X) = \LL(X)$.  This is not true for stacks, but we will now show that for a schematic map $f: X \rightarrow Y$, the formal and unipotent loops of $X$ are loops in $X$ whose images in $Y$ are formal and unipotent respectively.

\begin{prop}\label{FormalLoopsQuotient}
	Suppose that $f: X \rightarrow Y$ is a map of algebraic stacks representable by schemes.  Then,
	$$\LLf(X) = \LLf(Y) \times_{\LL(Y)} \LL(X).$$
\end{prop}
\begin{proof}
	It suffices to show that the closed classical substack $\pi_0(Y \times_{\LL(Y)} \LL(X))$ has the same reduced points as $X \subset \LL(X)$.  To do this, it suffices to check on geometric points.  
	Consider the diagram of \emph{classical} pullbacks
	$$\begin{tikzcd}
	\Omega(X, x) \arrow[r] \arrow[d] & \Omega(Y, f(x)) \arrow[r] \arrow[d] & \Spec k \arrow[d, "x"] \\
	\Spec k \arrow[r, "x"]&f^{-1}(f(x)) \arrow[d]  \arrow[r] & X \arrow[d] \\
	&\Spec k \arrow[r, "f(x)"] & Y.
	\end{tikzcd}$$
	Since $f$ is schematic, $x: \Spec k \rightarrow f^{-1}(f(x))$ is a closed embedding of schemes (since it as a map of schemes admitting a retract; see Lemma \ref{FormalLoopsSchematic}).  Since $Y$ is an algebraic stack, $\Omega(Y, f(x)))$ is a scheme, and so $\Omega(X, x) \rightarrow \Omega(Y, f(x))$ is a closed embedding of schemes and a map of affine (classical) group schemes.  The preimage of the identity is thus the identity, so constant loops in $\LL Y$ are preimages of constant loops in $\LL X$.
\end{proof}

\begin{exmp}[Quotient stacks]
	In the case of quotient stacks, we have a map $\LL(X/G) \rightarrow \LL(BG) = G/G$.  The above proposition says that $\LLf(X/G)$ is the completion of $\LL(X/G)$ at the closed substack of points lying over $\{e\}/G \subset G/G$.
\end{exmp}




The following is well-known, but we provide a brief argument for the reader's convenience.
\begin{prop}\label{MapBG}
Let $H, G$ be groups in affine derived schemes over $k$.  There is a natural identification
$$\Map(BH, BG) = \Hom_{grp}(H, G)/G$$
where $G$ acts on $\Hom_{grp}(H, G)$ by the adjoint action.
\end{prop}
\begin{proof}
Note that we assume $k$ has characteristic zero.  
We define maps $\Phi: \Map(BH, BG) \rightarrow \Hom_{grp}(H, G)/G$ and $\Psi: \Hom_{grp}(H, G)/G \rightarrow \Map(BH, BG)$ and leave the verification that they are strict inverses to the reader.  These are maps of sheaves; we will restrict our attention to defining their map on $S$-points $\Phi_S$ and $\Psi_S$.

The map $\Phi_S$ of spaces is defined as follows.  Let $F \in \Map_S(BH \times S, BG \times S)$ be the map of sheaves whose value at $S' \rightarrow S$ is a functor $F_{S'}$ from (right) $H$-torsors over $S'$ to $G$-torsors over $S'$.  We define
$$\Phi_S(F)(S') = \left\{ \begin{tikzcd}
F_{S'}(H \times S') \arrow[r, "\phi"] \arrow[d, "G-\text{torsor}"] &  \Hom_{S-grp}(H \times S, G \times S) \\
S' \end{tikzcd}\right\}
$$
where $\phi$ is defined as follows.  There is a canonical identification of automorphisms of the trivial torsor $\Aut_{S'}(S' \times H) = H(S')$, and for $h \in \Aut_S(S \times H)$, $F(h)$ is an automorphism of $F(S' \times G)$, which we abusively write $F(h) \in G$ as an section in $G(S')$.  We define $\phi(x)(h) =x \cdot F(h)$.  

The map $\Psi$ is defined as follows.  The $S$-points of $\Hom_{grp}(H, G)/G$ are $G$-torsors $P$ over $S$ with $G$-equivariant maps $\phi: P \rightarrow \Hom_{grp}(H, G)$.  We define for $S' \rightarrow S$ and $Q$ a $H$-torsor over $S'$
$$\Psi_S(P, \phi)(S')(Q) = Q \times^H (P \times_S S')$$
where $H$ acts on $P$ on the left via $\phi$.
\end{proof}

\begin{defn}
A map of prestacks $X \rightarrow Y$ is a \emph{monomorphism}, i.e. $X$ is a \emph{substack} of $Y$, if for any affine derived scheme $S$ and $y \in Y(S)$ the fiber product $\{y\} \times_{Y(S)} X(S)$ is contractible (in the category of spaces).
\end{defn}

%

\begin{prop}\label{UniMono}
Let $X$ be a geometric stack.  The map $\LL^u X \rightarrow \LL X$ is a monomorphism, i.e. unipotence of a loop is a property and not a structure.
\end{prop}
\begin{proof}
Let $S = \Spec(R)$.  Consider an $S$-point $\gamma \in (\LL X)(S)$, which determines a base point $x \in X(S)$ and a based loop $g \in \Omega(X, x)(S)$.  We wish to show that $\LL^u X(S) \times_{\LL X(S)} \{(\gamma)\}$ is contractible.  Equivalently, we wish to show that $\Omega^u(X, x)(S) \times_{\Omega(X, x)(S)} \{g\}$ is contractible.  Since $X$ is geometric, $\Omega(X, x)$ is derived affine.  By Lemma \ref{UnipotentLoopsDescription} it suffices to show that for a derived affine group $G$ over $S$, the map of stacks $\Hom_{grp,S}(\G_a \times S, G) \rightarrow G$ is a monomorphism on $S$-points.  Note that an $\infty$-group object is by definition an $\infty$-monoid object satisfying a condition; in particular, the map $\Hom_{grp,S}(\G_a \times S, G) \rightarrow \Hom_{mon,S}(\G_a \times S, G)$
is a monomorphism.  Furthermore, the forgetful functor from $\Coalg(\Alg(R))$ to $\Coalg(R\dmod) = \Coalg(R)$ is fully faithful, so the map $\Hom_{mon,S}(\G_a \times S, G) \rightarrow \Hom_{\Coalg(R)}(p_* \OO_{\Omega(X, x)}, \OO_{S \times \G_a})$ is a monomorphism. It suffices to show that the map induced by evaluation at $1 \in \G_a$:
$$\Hom_{\Coalg(R)}(p_* \OO_{\Omega(X, x)}, \OO_{S \times \G_a}) = \Hom_{\Coalg(R\dmod)}(R[\Omega(X, x)], R[x]) \rightarrow \Hom_{R}(R[\Omega(X, x)], R)$$
is also a monomorphism.  By the calculation in Lemma 1.12 of \cite{GG:DGC}, $\OO_S \otimes_k \OO(\G_a)$ is the cofree coalgebra object in $\Coalg(\QCoh(S))$, so this map is an equivalence.
\end{proof}

\begin{prop}[Unipotent loops of quotient stacks]\label{UnipotentLoopsQuotient}
Let $G$ be a classical affine algebraic group over $k$.  Then,
$$\LL^u(BG) = \widehat{U}/G$$
where $U$ is the unipotent cone of $G$ (i.e. the closed subvariety of unipotent elements of $G$).   Furthermore, if $G$ acts on a scheme $X$, then $\LL^u(X/G)$ is computed by the pullback square
$$\begin{tikzcd}
\LL^u(X/G) \arrow[d] \arrow[r] & \LL(X/G) \arrow[d] \\
\widehat{U}/G \arrow[r] & G/G.
\end{tikzcd}$$
\end{prop}
\begin{proof}
It suffices to prove the statement on based loops at a given $x \in (X/G)(S)$ where $S$ be an affine derived scheme, say $S = \Spec(R)$.  The based loops in $X/G$ for any $S$-point $x \in X(S)$ can be computed via a fiber product
$$\begin{tikzcd}
\Omega(X/G, x) \arrow[r] \arrow[d] & \Omega(BG, p(x)) = G \times S \arrow[d] \\
S \arrow[r, "x"] & X.
\end{tikzcd}$$
Let $U_G$ denote the unipotent cone of $G$, and let $U_{\Omega(X/G, x)}$ denote the closed reduced classical subscheme of the inverse image of $U_G \times S \subset G \times S$.  Let $\widehat{U}_{\Omega(X/G, x)}$ the formal completion of $\Omega(X/G, x)$ at $U_{\Omega(X/G, x)}$.  We first claim that $\Omega^u(X/G, x) = \widehat{U}_{\Omega(X/G, x)}$.

Using Proposition \ref{UnipotentLoopsDescription}, evaluation at $1 \in \G_a$ provides a map $\Phi: \Omega^u(X/G, x) \rightarrow \Omega(X/G, x)$, which is a monomorphism by Proposition \ref{UniMono}.  Such a map factors through $\widehat{U}_{\Omega(X/G, x)}$ if its (classical) set-theoretic image lies in the inverse image of the unipotent cone $U \times S$.  This, we can check on $k$-points of $S$, and in particular assume that $S = k$.  The claim now follows from the classical result that for a map of linear algebraic groups over $k$, the image of a unipotent element must also be unipotent (i.e. its eigenvalues are all $1 \in k$).  In particular, $\Omega^u(X/G, x) \subset \widehat{U}_{\Omega(X/G, x)} \subset \Omega(X/G, x)$.

For surjectivity, we define the inverse map $\Psi: \widehat{U}_{\Omega(X/G, x)} \rightarrow \Omega^u(X/G, x) = \Hom_{grp, S}(\G_{a, S}, \Omega(X/G, x))$ via the adjoint to an exponential map $\G_{a, S} \times_S \widehat{U}_{\Omega(X/G, x)} \rightarrow \Omega(X/G, x)$ which we will now construct.  We take as a given that such an exponential map is constructed for classical affine algebraic groups over $k$, i.e. we have a map $\G_a \times \widehat{U}_{G} \rightarrow G$.  To define an exponential map for $\Omega(X/G, x)$, we use the universal property of fiber products and the classical fact that if an $S$-point $g \in G(S)$ fixes $x \in X(S)$, then so does $g^t \in G(S)$ for $t \in \G_a(S)$.  More precisely, the following diagram commutes, inducing the desired exponential map
$$\begin{tikzcd}
\G_{a,S} \times_S \widehat{U}_{\Omega(X/G, x)} \arrow[r] \arrow[dr, dotted] \arrow[ddr] & \G_{a,S} \times \widehat{U}_G \arrow[dr, "\exp"]  \\
& \Omega(X/G, x) \arrow[r] \arrow[d] & G \times S \arrow[d] \\
& S \arrow[r, "x"] & X.
\end{tikzcd}$$
\end{proof}

\subsection{Cyclic homology}\label{HHSect}

In this section we give a brief overview of the basic definitions of Hochschild homology and cyclic homology, as developed in Section 5.5 of \cite{Lu:HA}.  Further discussion can be found in \cite{BZN:NT} \cite{BZN:NT} \cite{Ho:HH} \cite{AMGR:FHE} \cite{AF:CH} \cite{AFR:FH1} \cite{NS}.

\begin{defn}
Let $\cat{Cat}_\otimes$ be a symmetric monoidal $\infty$-category with monoidal unit $1_{\otimes}$, and $X \in \cat{Cat}_{\otimes}$ a 1-dualizable object with dual $X^\vee$, coevaluation $\eta: 1_{\otimes} \rightarrow X \otimes X^\vee$ and evaluation $\epsilon: X^\vee \otimes X \rightarrow 1_{\otimes}$.  We define the \emph{dimension} of $X$ by
$$\dim(X) = \epsilon \circ \eta \in \cat{End}_{\cat{Cat}_{\otimes}}(1_{\otimes}).$$ 
If $F: X \rightarrow Y$ is a morphism with a right dual (i.e. adjoint) $G$, then we can define
$$\dim(F): \begin{tikzcd}
\dim(X) \ar[r, "\gamma"] & \text{tr}(G \circ F) \ar[r, "\simeq"] & \text{tr}(F \circ G) \ar[r, "\upsilon"] & \dim(Y).
\end{tikzcd}$$
\end{defn}

\begin{rmk}
More precisely, the a choice of dualizing structure for $X \in \cat{Cat}_{\otimes}$ determines an explicit dimension $\dim(X)$.  By Proposition 4.6.1.10 in \cite{Lu:HA}, the space of dualizing structures on $X$ is contractible; therefore, $\dim(X)$ is defined uniquely up to unique isomorphism in the homotopy category $\Ho(\cat{End}_{\cat{Cat}_\otimes}(1_\otimes))$.  This fact allows us to compute Hochschild homology in two different ways using different dualizing structures on a category and know that they are equivalent without explicitly producing an equivalence.
\end{rmk}

\begin{rmk}
Lurie's proof of the Cobordism Hypothesis \cite{Lu:TFT} allows for an equivalent formulation: there is a bijection between $1$-dualizable objects $X \in \cat{Cat}_{\otimes}$ and framed extended $\cat{Cat}_{\otimes}$-valued $n$-dimensional topological field theories $\mathcal{Z}_X$; for a 1-dualizable object $X$ we define the dimension by
$$\dim(X) = \mathcal{Z}_X(S^1).$$
By this definition, there is evidently an $S^1$-action on $\dim(X)$.  The relationship between $S^1$-action and its explicit realization via the cyclic structure is spelled out in Theorem 5.5.3.11 in \cite{Lu:HA}.
\end{rmk}


We will define the Hochschild homology of a category to be its dimension; we first need to define a monoidal structure on $\infty$-categories.
\begin{defn}[Lurie tensor product]
The category $\cat{Pr}^L$ is equipped with a monidal structure called the \emph{Lurie tensor product} in Proposition 4.8.1.15 of \cite{Lu:HA}, which can be thought of as an $\infty$-analogue of the Deligne tensor product.  
It is equipped with a canonical functor 
$$\cat{C} \times \cat{D} \rightarrow \cat{C} \otimes \cat{D} \;\;\;\;\;\;\;\;\;\; (X, Y) \mapsto X \boxtimes Y.$$
which is initial amongst functors out of $\cat{C} \times \cat{D}$ which preserve small colimits separately in each variable.
\end{defn}

\begin{rmk}
Proposition 4.8.1.17 of \cite{Lu:HA} provides an explicit realization $\cat{C} \otimes \cat{D} \simeq \cat{Fun}^R(\cat{C}^{op}, \cat{D})$ which is presentable by Lemma 4.8.1.16 of \cite{Lu:HA}.  In particular, by \cite{Lu:HTT} Proposition 5.5.3.8, the Lurie tensor product makes $\cat{Pr}^L$ into a closed monoidal category with internal mapping object $\cat{Fun}^L(-, -)$.  Furthermore, by Propositions 4.8.2.10 and 4.8.2.18 in \cite{Lu:HA}, the Lurie tensor product induces a tensor product on $k$-linear presentable categories $\cat{Pr}^L_k$.
\end{rmk}

\begin{defn}
Let $\cat{Pr}^{L, \omega}_{k, \vee}$ be the $\infty$-category of dualizable presentable stable $k$-linear $\infty$-categories, and functors which preserve compact objects.  We define the \emph{Hochschild homology} functor by
$$HH := \dim: \cat{Pr}^{L, \omega}_{k, \vee} \rightarrow \cat{Fun}^L_k(\cat{Vect}_k, \cat{Vect}_k) \simeq \cat{Vect}_k.$$
By \cite{AFR:FH1}, the Hochschild homology as defined above has an $S^1$-action.
\end{defn}

\begin{rmk}
The dimension of a dualizable $k$-linear category takes values in chain complexes under the equivalence
$$\cat{End}_{\cat{Cat}_{\otimes}}(1_{\otimes}) = \cat{Fun}^L_k(\cat{Vect}_k, \cat{Vect}_k) \simeq \cat{Vect}_k.$$  
That is, every such endofunctor $F$ commuting with colimits is determined by its value $F(k)$.  In particular, a choice of dualizing structure determines an explicit Hochschild chain complex.
\end{rmk}

\begin{rmk}
Note that the Lurie tensor product is defined on $\cat{Pr}^L$, and dimension is functorial only for right dualizable maps, so Hochschild homology is only functorial for functors whose right adjoints also admit right adjoints (equivalently, whose right adjoints are also continuous).  By Proposition 5.5.7.2 of \cite{Lu:HTT}, these are exactly the functors which preserve compact objects.
\end{rmk}

Traditionally, Hochschild homology is formulated in the setting of small stable $k$-linear dg categories.  In particular, in this setting it is possible to explicitly write out a bar complex computing the Hochschild homology.  We will see that the above ``large'' notion of Hochschild homology is a generalization of the ``small'' version.  The following is proven in Theorem D.7.0.7 in \cite{Lu:SAG} and Proposition 4.6.15 in \cite{Lu:HA}.
\begin{prop}
If $\cat{C} \in \cat{Pr}^L_k$ is compactly generated, then it is dualizable.  In particular, if $\cat{C} = \Ind(\cat{C}^0)$, then $\cat{C}^\vee = \Ind(\cat{C}^{0, op})$, and the evaluation map is given by ind-completion via universal properties of the Yoneda pairing $\Hom(-, -): \cat{C}^{0, op} \times \cat{C} \rightarrow \cat{Vect}_k$.  Furthermore, there are natural isomorphisms
$$\cat{Fun}^L_k(\cat{C}, \cat{C}) \simeq \cat{Fun}^L_k(\cat{C}, \cat{Vect}_k) \otimes \cat{C} \simeq \cat{C}^\vee \otimes \cat{C}$$
which realize the coevaluation via the identity functor in $\cat{Fun}^L_k(\cat{C}, \cat{C})$.  
\end{prop}

\begin{defn}
Let $\cat{st}_k$ be the $\infty$-category of small stable $k$-linear $\infty$-categories.  We define the \emph{Hochschild homology} functor
$$HH := \dim \circ \Ind: \cat{st}_k \rightarrow \cat{Fun}^L_k(\cat{Vect}_k, \cat{Vect}_k) \simeq \cat{Vect}_k$$
i.e. $HH(\cat{C})$ is the image of $k$ under the composition
$$\begin{tikzcd}
\cat{Vect}_k \ar[r, "\text{coev}"] & \Fun^L_k(\cat{C}, \cat{C}) \ar[r, "\simeq"] & \cat{C}^\vee \otimes \cat{C} \ar[r, "\text{ev}"] & \cat{Vect}_k.
\end{tikzcd}$$
\end{defn}

\begin{rmk}
If $\cat{C}$ is compactly generated, then $HH(\cat{C}) = HH(\cat{C}^\omega)$, so the two definitions of Hochschild homology given above are compatible.
\end{rmk}

We now seek to understand the $S^1$-action on Hochschild homology.  While there are purely categorical ways to view $S^1$ actions, we model the concretely on chain complexes via a mixed differential.

\begin{defn}
A \emph{mixed complex} is a dg-module over the dg-algebra $H_\bullet(S^1; k) \simeq k[\epsilon]$ where $|\epsilon| = -1$.  Explicitly, it is a chain complex $(V, d)$ with a \emph{mixed differential} $\epsilon$ of cohomological degree $-1$ such that $d\epsilon = \epsilon d$ and $\epsilon^2 = 0$.
\end{defn}

We are interested in the following operations on mixed complexes.
\begin{defn}
We define the \emph{$S^1$-invariants} and \emph{$S^1$-coinvariants functors} by
$$V^{S^1} := R\Hom_{C_\bullet(S^1; k)}(C_\bullet(ES^1; k), V) \simeq (V[[u]], d + u \epsilon) \in k[[u]]\dmod,$$
$$V_{S^1} := V \otimes_{C_\bullet(S^1; k)} C_\bullet(ES^1; k) \simeq (V[u^{-1}], d + u\epsilon) \in k[[u]]\dmod_{u-\text{tors}}.$$
and the \emph{Tate construction} by
$$V^{\Tate} = V^{S^1} \otimes_{k[[u]]} k((u)) = \lim u^k V_{S^1} =  (V((u)), d + u\epsilon) \simeq (V((u)), d + u\epsilon) \in k((u))\dmod.$$
\end{defn}

\begin{rmk}
The action of $C_\bullet(S^1; k)$ on $C_\bullet(ES^1; k)$ comes from the sweep action of chains of \cite{GKM:KD}.  
Taking a presentation of $ES^1$ as a colimit of odd spheres with free $S^1$-actions, the $S^1$-invariants can be expressed as a filtered limit
$$V^{S^1} = \lim_n R\Hom_{C_\bullet(S^1; k)}(C_\bullet(S^{2n+1}; k), V) \simeq  \lim_n \, (V[u]/u^n, d + u\epsilon).$$
\end{rmk}

\begin{rmk}
Note that the $S^1$-invariants operation (as well as the Tate construction)
is not continuous as defined above\footnote{For example, consider $N = \colim_n k[-2n]$.  Since $N$ is acyclic, $N^{S^1} = 0$.  On the other hand, $(k[-2n])^{S^1} = u^{-1}k[[u]]$ so that $\colim (k[-2n])^{S^1} = k((u))$.}, but can be made so by considering mixed complexes in the category $\Ind(k[\epsilon]\dmod_{f.g.})$ instead.   
\end{rmk}

\begin{defn}
We respectively define the \emph{negative cyclic homology}, \emph{cyclic homology} and \emph{periodic cyclic homology} by
$$HN(\cat{C}) := HH(\cat{C})^{S^1}, \;\;\;\;\; HC(\cat{C}) := HH(\cat{C})_{S^1}, \;\;\;\;\; HP(\cat{C}) := HH(\cat{C})^{\Tate}.$$
\end{defn}

In this note we consider two different explicit models of Hochschild homology and its cyclic variants.  One is the usual cyclic bar construction on a small dg category, and the other is via $S^1$-equivariant functions on the derived loop space.

\begin{exmp}[Algebraic model]
Let $A$ be a dg algebra (or more generally, a dg category) over $k$, and $\cat{C} = A\dmod$ the category of left dg-modules over $A$.  For two dg categories $A, B$, an $A^{op} \otimes_k B$-module defines a continuous functor $A\dmod \rightarrow B\dmod$.  By the dg Morita theory of \cite{To:DG}, this functor
$$A^{op} \otimes B\dmod \rightarrow \Fun^L_k(A\dmod, B\dmod)$$
is an equivalence.  Under this equvalence, the coevaluation $k\dmod \rightarrow A \otimes A^{op}\dmod$ corresponds to the functor $A \otimes_k -$, where $A$ is considered as a bimodule over itself, and the evaulation map corresponds to $- \otimes_{A \otimes A^{op}} A$.  
In particular, the Hochschild homology is given by the usual Hochschild homology
$$HH(A\dmod) = HH(A; A) = A \otimes_{A \otimes A^{op}} A.$$
\end{exmp}

\begin{defn}[Cyclic bar complex]\label{ExmpHHdgcat}
The \emph{cyclic nerve} of a small $k$-linear dg category $\cat{C}$ is the cyclic vector space whose $n$-simplices are given by
$$C_n(\cat{C}) = \coprod_{X_0, \ldots, X_n \in \mathrm{Ob}(\cat{C})} \Hom_{\cat{C}}(X_0, X_n) \otimes \Hom_{\cat{C}}(X_n, X_{n-1}) \otimes \cdots \otimes \Hom_{\cat{C}}(X_1, X_0)$$
where the face maps are given by composition, the degeneracy maps by the identity homomorphism, and the cyclic structure by rotation of the terms.  Its associated chain complex, which also abusively denote by $C_\bullet(\cat{C})$, is the \emph{cyclic bar complex} which is naturally a mixed complex with the mixed differential arising via the Connes $B$-operator.  This mixed differential exhibits the $S^1$-action on Hochschild homology \cite{Ho:HH}.
\end{defn}

\begin{rmk}
One can obtain smaller models by taking objects from a set of compact generators (see Theorem 5.2 of \cite{Ke:DG}) rather than all of $\mathrm{Ob}(\cat{C})$; for example, if $\cat{C} = A\dperf$, then the free module $A$ is a compact generator and one recovers the classical cyclic bar complex $C^\bullet(A; A)$.
\end{rmk}


\begin{exmp}[Geometric model]
By \cite{BZFN:IT}, when $X$ is a perfect stack (e.g. a quotient stack of a derived quasiprojective scheme by an affine group in characteristic zero), then $\QCoh(X)$ is compactly generated by $\Perf(X)$ and we have isomorphisms
$$\QCoh(X) \otimes \QCoh(X) \simeq \QCoh(X \times X) \simeq \Fun^L_k(\QCoh(X), \QCoh(X))$$
where the functors on the right are given by integral transforms.  Explicitly, we identify 
$$\begin{tikzcd} \QCoh(X) \ar[r, "\simeq"] &  \QCoh(X)^\vee\end{tikzcd}$$
 on compact objects $\mathcal{K} \in \Perf(X)$ by
$$\mathcal{K} \mapsto \Gamma(X, \mathcal{K} \otimes -) \simeq \Gamma(X, \shHom_X(\mathcal{K}^\vee, -)).$$
Letting $p: X \rightarrow \Spec(k)$ be the map to a point and $\Delta: X \rightarrow X \times X$ the diagonal, the coevaluation is given by the functor $\Delta_* p^*$ and the evaluation by $p_* \Delta^*$.  In particular, we find that the Hochschild homology is
$$HH(\Perf(X)) \simeq p_* \Delta^* \Delta_* p^* k = \Gamma(X, \Delta^* \Delta_* \OO_X) \simeq \Gamma(\LL X, \OO_{\LL X}) = \OO(\LL X)$$
with the last isomorphism arising via base change.  The $S^1$-action on $\OO(\LL X)$ is the $S^1$-equivariant structure arising from loop rotation; for details see Remark 3.2 and Proposition 4.2 in \cite{BZN:NT}.
\end{exmp}

\begin{rmk}
If $X$ is QCA but not perfect, it is not currently known whether $\QCoh(X)$ is compactly generated.  On the other hand, by Theorem 4.3.1 of \cite{DG:QCA}, $\QCoh(X)$ is dualizable, so that
$$HH(\QCoh(X)) = \OO(\LL X)$$
by a similar argument.  It does not appear to be known whether $HH(\Perf(X))$ agrees with $HH(\QCoh(X))$.
\end{rmk}

\subsection{The equivariant cyclic bar construction}

We now define an explicit model for the Hochschild homology of a quotient stack.
\begin{prop}\label{GeneratorQuotientStack}
Let $X$ be a quasiprojective scheme with an action of a reductive group $G$, and let $p: X/G \rightarrow BG$ and $q: X \rightarrow X/G$ be the natural maps of stacks.  Let $\mathcal{E} \in \Perf(X/G)$ be a locally free sheaf such that $q^* \mathcal{E}$ is a compact generator of $\QCoh(X)$.  Define
$$A = p_*R\shHom_X(\mathcal{E}, \mathcal{E}) \in \Alg(\QCoh(BG)).$$
Then, the functor
$$R\Hom_X(\mathcal{E}, -): \QCoh(X/G) \rightarrow A\dmod_{\QCoh(BG)}$$
is an equivalence of dg categories.
\end{prop}
\begin{proof}
Since $X$ is quasiprojective, it admits an equivariant compact generator $\mathcal{E}$ of $\Perf(X)$ (not $\Perf(X/G)$).  The functor is fully faithful since $q$ is an atlas and $q^*\mathcal{E}$ is a generator.  It is essentially surjective, since $A\dperf_{\QCoh(BG)}$ is generated by $A \otimes V$ for $V \in \Irr(G)$, and $R\Hom_{X/G}(\mathcal{E}, \mathcal{E} \otimes V) = A \otimes V$.
\end{proof}

We need the following formula for $G$-representations.
\begin{prop}\label{BGProjForm}
 Let $G$ be a reductive affine algebraic group, and $V, W$ two rational representations of $G$ (i.e. $V, W \in \QCoh(BG)$).  Then, there is a natural equivalence
$$\begin{tikzcd}\displaystyle\bigoplus_{U \in \Irr(G)} (V \otimes U^*)^G \otimes (U \otimes W)^G \arrow[r, "\simeq"] & (V \otimes W)^G .\end{tikzcd}$$
More generally, for $V_0, \ldots, V_n$ rational representations of $G$, there is a natural equivalence
$$
\begin{tikzcd}
\displaystyle\bigoplus_{U_1, \ldots, U_n \in \Irr(G)} (V_0 \otimes U_1^*)^G \otimes (U_1 \otimes V_1 \otimes U_2^*)^G \otimes \cdots \otimes (U_n \otimes V_n)^G \arrow[r, "\simeq"] & (V_0 \otimes \cdots \otimes V_n)^G.\end{tikzcd}$$
\end{prop}
\begin{proof}
Consider $V \boxtimes W \in \QCoh(BG \times BG)$ as a $G \times G$-representation.  The projection formula defines an equivalence
$$\Gamma(BG \times BG, \Delta_* \OO_{BG} \otimes (V \boxtimes W)) \rightarrow \Gamma(BG, \Delta^*(V \boxtimes W)).$$
The result follows by unwinding this equivalence using the Peter-Weyl theorem for reductive algebraic groups $\Delta_* \OO_{BG} = k[G] = \bigoplus_{U \in \Irr(G)} U^* \otimes U$, and noting that taking global sections amounts to taking $G$ or $G \times G$-invariants (note that there are no higher cohomology groups since $G$ is reductive).  The second claim results from iterating the first, or by applying the projection formula to the diagonal $BG \rightarrow BG^{\times n}$.
\end{proof}

\begin{rmk}
 The map above, which we now denote $\phi$, and its inverse $\psi$ can be written out explicitly:
 $$\phi(v_0 \otimes u_1^* \otimes u_1 \otimes v_1 \otimes u_2^* \otimes \cdots \otimes u_n \otimes v_n) = \left( u_1^*(u_1) \cdots u_n^*(u_n) \right) \cdot v_0 \otimes \cdots \otimes v_n$$
 $$\psi(v_0 \otimes \cdots \otimes v_n) = \pi\left(\sum_{U_1, \ldots, U_n \in \Irr(G)} v_0 \otimes \eta_1 \otimes v_1 \otimes \eta_2 \otimes \cdots \otimes \eta_n \otimes v_n\right)$$
 where $\eta_i \in U_i^* \otimes U_i$ is the identity map, and $\pi = (-)^G \otimes \cdots \otimes (-)^G$ is the tensor products of the projections to the $G$-invariant isotypic component.
\end{rmk}

\begin{defn}
Let $A \in \Alg(\QCoh(BG))$ be an associated algebra object with coaction map $c: A \rightarrow A \otimes k[G]$.  The \emph{equivariant cyclic bar complex} $C_\bullet(A, G)$, defined in \cite{BG:ECH}, is the mixed complex associated to the following cyclic vector space.  We define the $n$-simplices by
$$C_n(A, G) = (A^{\otimes n+1} \otimes k[G])^G$$
with face and degeneracy maps
$$d_i(a_0 \otimes \cdots \otimes a_n \otimes f) = a_0 \otimes \cdots \otimes a_i a_{i+1} \otimes \cdots \otimes a_n \otimes f\;\;\;\;\;\;\; i=0, \ldots, n-1,$$
$$d_n(a_0 \otimes \cdots \otimes a_n \otimes f) = c(a_n) a_0 \otimes \cdots \otimes f,$$
and cyclic structure
$$t(a_0 \otimes \cdots \otimes a_n \otimes f) = c(a_n) \otimes a_0 \otimes \cdots \otimes a_{n-1} \otimes f,$$
$$s_{n+1}(a_0 \otimes \cdots \otimes a_n \otimes f) = 1 \otimes a_0 \otimes \cdots \otimes a_n \otimes f.$$
\end{defn}

\begin{rmk}
Note that it is essential to take $G$-invariants for $C_\bullet(A, G)$ to admit a cyclic sructure.  In particular, if $a_0 \otimes \cdots \otimes a_n \otimes f$ is $G$-invariant, then 
$$c(a_0) \otimes \cdots \otimes c(a_n) \otimes f = a_0 \otimes \cdots \otimes a_n \otimes f.$$
\end{rmk}

\begin{prop}
The equivariant cyclic bar complex $C_\bullet(G, A)$ computes $HH(\Perf(A\dmod_{\QCoh(BG)})).$
\end{prop}
\begin{proof}
Since $A$ generates $A\dperf$, the objects $A \otimes V_\lambda$ generate $A\dperf_{\QCoh(BG)}$ and we have the following ``small'' model for $HH(A\dperf_{\QCoh(BG)})$.  The cyclic bar complex corresponding to the generating set $\{A \otimes V \mid V \in \Irr(G)\}$ of $A\dperf_{\QCoh(BG)}$ has terms
$$D_n(A, G) = \displaystyle\bigoplus_{V_0, \ldots, V_n \in \Irr(G)} \Hom_A(A \otimes V_0, A \otimes V_1)^G \otimes \Hom_A(A \otimes V_1, A \otimes V_2)^G \otimes \cdots \otimes \Hom_A(A \otimes V_{n-1}, A \otimes V_n)^G.$$
Using the natural isomorphism $\Hom_A(A \otimes V, A \otimes W)^G = (A \otimes V^* \otimes W)^G$, we rewrite:
$$D_n(A, G) = \displaystyle\bigoplus_{V_0, \ldots, V_n \in \Irr(G)} (V_0^* \otimes A \otimes V_1)^G \otimes (V_1^* \otimes A \otimes  V_2)^G \otimes \cdots \otimes ( V_n^* \otimes A \otimes V_0)^G.$$
Next, we can rewrite $C^n(A, G)$ using the Peter-Weyl theorem for reductive algebraic groups:
$$C_n(A, G) = \bigoplus_{V_0 \in \Irr(G)} (V_0 \otimes^* A^{\otimes n} \otimes V_0)^G.$$
Applying Proposition \ref{BGProjForm} to each summand (i.e. for fixed $V_0$) produces an equivalence $D_n(A, G) \rightarrow C_n(A, G)$; the claim that it defines a map of cyclic objects is left to the reader.
\end{proof}

\begin{rmk}\label{EquivariantCyclicProjective}
The equivariant cyclic bar complex $C_\bullet(X, G)$ provides us with an explicit model of the Hochschild homology $HH(\Perf(X/G))$ as a $k[G]^G$-linear mixed complex (i.e. both the internal and mixed differentials are $k[G]^G$-linear).  Furthermore, since $(-)^G$ is a left adjoint functor whose right adjoint preserves epimorphisms, the terms in $C_\bullet(A, G)$ are also projective.
\end{rmk}

\subsection{Formal and derived completions}\label{FormalComplSection}

We review the notion of formal completions of derived stacks and the notion of derived completion in the derived category. This section is essentially a summary of the results in Chapter 4 of \cite{Lu:DAGXII}, Section 3.4 of \cite{BS:PE}, Section 15.80 of the \cite{SP}, and Chapter 6 of \cite{GR:DGI}.

\begin{defn}
Let $f: X \rightarrow Y$ be a map of derived stacks (or more generally, prestacks).  The \emph{formal completion} of $f$, written $\widehat{Y_X}$, is a prestack whose functor-of-points whose $S$-points are given by diagrams
$$\begin{tikzcd}
\pi_0(S)^{red} \arrow[r] \arrow[d] & X \arrow[d] \\
S \arrow[r] & Y.
\end{tikzcd}$$
It can also be defined via the fiber product
$$\widehat{Y_X} := Y \times_{Y^{\dR}} X^{\dR}.$$
\end{defn}

\begin{lemma}\label{FormalNbhd}
The formal completion of a map $X \rightarrow Y$ only depends on $\pi_0(X)^{red} \rightarrow Y$.  In particular, if $Z \rightarrow Y$ is a closed embedding, then $\widehat{Y_Z} \times_Y X = \widehat{X_{\pi_0(Z \times_Y X)^{red}}}$.
\end{lemma}
\begin{proof}
This follows directly from the functor-of-points characterization of formal completions, and that $\pi_0(X)^{red}$ is the universal stack that canonically factors any map from a classical reduced scheme $\pi_0(S)^{red}$.
\end{proof}

\begin{defn}
Let $f: A \rightarrow B$ be a map of derived rings.  Following \cite{Lu:DAGIV}, we say that $f$ is \emph{\'{e}tale} if the induced map $\pi_0(A) \rightarrow \pi_0(B)$ is \'{e}tale and for every $n \in \mathbb{Z}$, the map $\pi_n(A) \otimes_{\pi_0(A)} \pi_0(B) \rightarrow \pi_n(B)$ is an isomorphism of abelian groups.  A map of derived schemes is \'{e}tale if it is for a Zariski cover, and a representable map of derived Atin stacks is \'{e}tale if it is after base change to a cover.
\end{defn}

\begin{prop}\label{EtaleCotangent}
Let $X, Y, Z$ be stacks admitting deformation theory\footnote{This notion is defined in Chapter IV of \cite{GR:DAG} and is satisfied by quotient stacks.}.  Suppose that $f: X \rightarrow Y$ is \'{e}tale, and let $Z \rightarrow X$ be any map.  Then, the relative cotangent complex vanishes $\mathbb{L}_{X/Y} \simeq 0$.  Furthermore, the map on formal completions $\widehat{X}_Z \rightarrow \widehat{Y}_Z$ is an equivalence.
\end{prop}
\begin{proof}
The first sentence is Proposition 2.22 in \cite{Lu:DAGIV}.  For the second, by the exact triangle for cotangent complexes we have a natural isomorphism $\mathbb{L}_{Z/Y} \simeq \mathbb{L}_{Z/X}$ under $\mathbb{L}_Z$, and note that the formal completion of a map $Z \rightarrow X$ is a colimit of square-zero extensions controlled by the the map between cotangent complexes $\mathbb{L}_Z \rightarrow \mathbb{L}_{Z/X}$ (see Chapter IV.5 in \cite{GR:DAG}).
\end{proof}

We now discuss how to compute the derived completion of a quasicoherent complex in the derived category.  It is defined abstractly a the left adjoint to an inclusion functor of complete  objects which we now define.
\begin{defn}
Let $A$ be a connective dg ring, and fix an ideal $I \subset \pi_0(A)$.  We define the full subcategory $A\dmod_{nil}$ of \emph{$I$-nilpotent} objects consisting of those modules on which $I$ acts locally nilpotently, i.e. for each cycle $m \in H^\bullet(M)$ there is some power of $I$ which annihilates $m$.  By Proposition 4.1.12 and 4.1.15 in \cite{Lu:DAGXII}, the inclusion $A\dmod_{nil} \hookrightarrow A\dmod$ is continuous and preserves compact objects; therefore it has a continuous right adjoint $\Gamma_I$, which we call the \emph{local cohomology functor}.  We define the full subcategory $A\dmod_{loc}$ of \emph{$I$-local} objects as the right orthogonal to $A\dmod_{nil}$, and the full subcategory $A\dmod_{cpl}$ of \emph{$I$-complete} modules to be the right orthogonal to $A\dmod_{loc}$.  The subcategory $A\dmod_{cpl}$ has an equivalent characterization as those modules such that the homotopy (derived) limit 
$$\begin{tikzcd}
\cdots \arrow[r, "x"] & M \arrow[r, "x"] & M \arrow[r, "x"]&M
\end{tikzcd}$$ 
is zero for all $x \in I$.  By Proposition 4.2.2 of \cite{Lu:DAGXII}, the inclusion of the complete objects has a left adjoint, which we call the (derived) \emph{completion functor} and denote $\widehat{(-)}$.
\end{defn}

The following is Proposition 4.2.5 in \cite{Lu:DAGXII}.
\begin{prop}
The composition of left adjoints
$$\begin{tikzcd}
A\dmod_{nil} \arrow[r, hook] &  A\dmod \arrow[r, "\widehat{(-)}"] & A\dmod_{cpl}
\end{tikzcd}$$
is an equivalence.  Consequently, the composition of its right adjoints
$$\begin{tikzcd}
A\dmod_{cpl} \arrow[r, hook] &  A\dmod \arrow[r, "\Gamma"] & A\dmod_{nil}
\end{tikzcd}$$
is also an equivalence.
\end{prop}

\begin{exmp}
Let $A = k[x]$ and $I = (x)$, then $k[[x]]$ is $I$-complete, $k[x,x^{-1}]/k[x]$ is $I$-nilpotent, $k[x]$ is neither, and $k[x]/x^n$ is both.  Furthermore,
$$\Gamma_I(k[[x]]) = \Gamma_I(k[x]) = k[x,x^{-1}]/k[x][-1] \;\;\;\;\;\;\;\; \widehat{k[x,x^{-1}]/k[x]} = \widehat{k[x][1]} = k[[x]][1].$$
\end{exmp}

The derived completion and local cohomology functors can each be computed in two ways.  The following can be found as Propositions 15.80.10 and 15.80.17 in \cite{SP} and in a global form as Proposition 3.5.1 in \cite{BS:PE} and Proposition 6.7.4 in \cite{GR:DGI}.  The statements on local cohomology are well known (and which we will not use).
\begin{prop}\label{DerivedCompletion}
Assume $\pi_0(A)$ is a noetherian ring, and choose generators $f_1, \ldots, f_r$ of $I \subset \pi_0(A)$.  We define the complex
$$G^\bullet = \left( \begin{tikzcd} \pi_0(A) \arrow[r] & \prod \pi_0(A)[\frac{1}{f_i}] \arrow[r] & \prod_{i, j} \pi_0(A)[\frac{1}{f_i}, \frac{1}{ f_j}] \arrow[r] & \cdots \arrow[r] & \pi_0(A)[\frac{1}{f_1}, \ldots, \frac{1}{f_r}]\end{tikzcd}\right).$$
The derived completion of an $A$-module $M$ can be computed in two ways:
$$\widehat{M} = R\lim_n M \otimes_{\pi_0(A)}^L A/(f_1^n, \ldots, f_r^n) \simeq R\Hom_{\pi_0(A)}(G^\bullet, M).$$
Likewise, we can compute the local cohomology of $M$ by
$$\Gamma_I(M) = \colim_n R\Hom_{\pi_0(A)}(M/(f_1^n, \ldots, f_r^n), M) \simeq G^\bullet \otimes^L_{\pi_0(A)} M.$$
The latter formula is the calculation of local cohomology via a Cech resolution with supports on an affine scheme.
\end{prop}

\begin{rmk}
A theory of formal completions is described in \cite{GR:DGI} with the following notation.  Let $X$ be a dg scheme, and $Z \subset X$ a classical closed subscheme.  There is a functor $\widehat{i}^*: \QCoh(X) \rightarrow \QCoh(\widehat{X}_Z)$ with a fully faithful (continuous) left adjoint $\widehat{i}_?$ whose essential image is the category of quasicoherent sheaves supported on $Z$, and a fully faithful (non-continuous!) right adjoint $\widehat{i}_*$ whose essential image is the category of quasicoherent sheaves complete with respect to the ideal sheaf for $Z$.  There is an exact triangle of functors arising from the localization functor $\widehat{i}_? \widehat{i}^*$ (whose essential image is cocomplete):
$$\Gamma_Z = \widehat{i}_? \widehat{i}^* \rightarrow \text{id}_{\QCoh(X)} \rightarrow j_* j^* = (-)|_U.$$
In particular, the functor on the left is local cohomology, and the functor on the right is restriction to $U$.  The (non-continuous) functor $\widehat{i}_* \widehat{i}^*$ is the (derived) completion.
\end{rmk}

The following lemma is likely well-known, but we could not find a reference.
\begin{lemma}\label{LemmaDevissage}
Let $X$ be a derived scheme, $i: Z \subset X$ a closed subscheme and $j: U = X - Z \rightarrow X$ its complement. Let $\phi: \mathcal{F} \rightarrow \mathcal{G}$ be a map of quasicoherent sheaves on $X$.  If the derived completion $\widehat{\phi}_Z$ and the restriction $\phi|_U$ are isomorphisms, then $\phi$ is an isomorphism.
\end{lemma}
\begin{proof}
Using the above exact triangle, to show that $\phi$ is an isomorphism, it suffices to show that $\Gamma_Z(\phi)$ is an isomorphism, or equivalently, that $\Gamma_Z(\cone(\phi)) = 0$.  To this end, note that $\widehat{\cone(\phi)} = \wh{i}_* \wh{i}^* \cone(\phi) = 0$, and that $\wh{i}_*$ is fully faithful, so that $\wh{i}^*\cone(\phi) = 0$, so that $\Gamma_Z(\cone(\phi)) = \wh{i}_?\wh{i}^* \cone(\phi) = 0$.
\end{proof}

\begin{exmp}
The above is not true for non-derived completions.  For example, take $X = \mathbb{A}^1$, $Z = \{0\}$, and $\phi: 0 \rightarrow M = k[x,x^{-1}]/k[x]$ ($M$ is the module of distributions supported at zero).  Since $M$ is supported at zero, $M|_U = 0$, and since $x^k M = M$ for all $k$, $\widehat{M}_Z = 0$, but $\phi$ is not an isomorphism.  On the other hand, the derived completion of $M$ is $k[[x]][1]$.
\end{exmp}

%

\section{An equivariant localization theorem in derived loop spaces and Hochschild homology}\label{MainSection}

\subsection{Equivariant localization for derived loop spaces}\label{FormLoopSection}

The following construction defines a notion of formal and unipotent loops near a semisimple orbit of $G/G$, i.e. an adjoint closed $G$-orbit in $G$ consisting of semisimple elements.
\begin{defn}[$z$-formal and $z$-unipotent loops]
Let $X$ be a derived scheme, $G$ a reductive group acting on $X$, and $z \in G$ a semisimple element.  We let $G^z$ denote the centralizer of $z$, i.e. the $z$-fixed points under the adjoint action.  We define $Z = \{gzg^{-1} \mid g \in G\}$ to be the closed $G$-orbit containing $z$ and $U_z = \{gzug^{-1} \mid g \in G, u \in U \cap G^z\}$ to be its saturation (here, $U$ is the unipotent cone of $G$ and $G^z$ is the centralizer of $z$).  We define the \emph{$z$-formal} and \emph{$z$-unipotent loops} in $BG$ by
$$\LLf_z(BG) :=  \widehat{Z}/G\rightarrow \LL(BG) = G/G, \;\;\;\;\;\;\; \;\;\;\LL^u_z(BG) := \widehat{U_z}/G \rightarrow \LL(BG) = G/G.$$
For a quotient stack $X/G$, we define the \emph{$z$-formal} and \emph{$z$-unipotent loops} by the pullback squares:
$$\begin{tikzcd}
\LLf_z(X/G) \arrow[r] \arrow[d] & \LL(X/G) \arrow[d] & & \LL^u_z(X/G) \arrow[r] \arrow[d] & \LL(X/G) \arrow[d] \\
\LLf_z(BG) \arrow[r] & \LL(BG)& & \LL^u_z(BG) \arrow[r] & \LL(BG).
\end{tikzcd}$$

This construction is functorial in representable maps over $BG$.  Note that by Propositions \ref{FormalLoopsQuotient} and \ref{UnipotentLoopsQuotient}, $\LLf_e(X/G) = \LLf(X/G)$ and $\LL^u_e(X/G) = \LL^u(X/G)$, where $e \in G$ is the identity.
\end{defn}

\begin{prop}\label{EtaleProp}
Let $G$ be a reductive group, and $z \in G$ a semisimple element.  The map $\LL(BG^z) \rightarrow \LL(BG)$ is (Zariski) locally \'{e}tale at $z$, i.e. there is a Zariski open neighborhood of $Z/G \subset \LL(BG) = G/G$, and therefore a Zariski open neighborhood of $U_z/G \subset \LL(BG)$, on which the map is \'{e}tale.
\end{prop}
\begin{proof}
Let us recall the set-up of the \'{e}tale slice theorem as in \cite{Dr:LS}.  Let $G$ be a reductive group acting on an affine variety $X$, and $x \in X$ a closed point such that the stabilizer $Z_G(x)$ is reductive.  We define a map $\phi: X \rightarrow T_x(X)$  as follows.  Let $\mf{m}$ be the maximal ideal for $x \in X$; the quotient map to the cotangent space has a $Z_G(x)$-equivariant splitting $\mf{m}/\mf{m}^2 \rightarrow \mf{m}$ since $Z_G(x)$ is reductive, defining a map $\Sym_k(\mf{m}/\mf{m}^2) \rightarrow k[X]$.  Geometrically, this means choosing functions $f_1, \ldots, f_n \in k[X]$ vanishing at $x$ whose differentials generate the cotangent space at $x$, and defining the map $\phi: X \rightarrow T_x(X)$ by evaluation 
$$y \mapsto \sum f_i(y) \;  \left. \frac{d}{df_i}\right|_{y=x},$$
in a $Z_G(x)$-equivariant manner.  The \'{e}tale slice is the inverse image $\phi^{-1}(N)$ where $N \subset T_x(X)$ is any normal subspace to $T_x(G \cdot x) \subset T_x(X)$, and the theorem tells us that the map $G \times^{Z_G(x)} \phi^{-1}(N) \rightarrow X$ is \'{e}tale.

Specializing to our situation, where $G$ acts on itself by the adjoint action, we produce the $G^z$-equivariant map $\phi: G \rightarrow T_z(G)$ as follows.  Affine locally at $z$, we can choose generators $f_1, \ldots, f_n$ such that 
$$k[G]/(f_1, \ldots, f_r) = k[G^z],$$ and the vanishing of $df_1, \ldots, df_r$ cuts out $\g^z \subset T_z(G) \simeq \g z$ (i.e. the translation of $T_e(G) = \g$ by central $z$). 
Thus it suffices to show that $\mathfrak{g}^z$ is a normal subspace to $T_z(G \cdot z)$, since $\phi^{-1}(\mathfrak{g}^z) = G^z$ by construction.  On the other hand, we have a natural isomorphism $G \cdot z \simeq G/G^z$, inducing $T_z(G \cdot z) \simeq \g z/\g^z z$, which produces a splitting of $T_z(G \cdot z) \subset T_z(G)$ whose kernel is $\g^z$.  Explicitly, $z$ is semisimple and acts on $\g$, so $\g$ decomposes into $z$-eigenspaces; $\g^z$ is the trivial eigenspace and $\g/\g^z$ is isomorphic to the sum of all other eigenspaces.

Using the fact that the map $G^z \times^{G^z} G \rightarrow G$ is $G$-equivariant, and since for any $u \in U_z$ we have $z \in \overline{G \cdot u}$, it follows that every open set containing $Z$ also contains $U_z$.
\end{proof}

\begin{exmp}\label{LocalFailure}
Let $G$ be a simple reductive algebraic group and choose $z \in G$ regular semisimple.  Its centralizer is a torus $T$ and the map $G \times^T T^{reg} \rightarrow G^{rs}$ is \'{e}tale with fiber $W_T = N(T)/T$.
\end{exmp}

\begin{cor}\label{UniFormBGzBG}
Let $X$ be a prestack equipped with the action of a reductive group $G$ over $k$, and let $z \in G$ be central.  Then, $\LL(X/G^z) \rightarrow \LL(X/G)$ is \'{e}tale over a neighborhood of $U_z/G \subset G/G$.  In particular, the natural maps $\LLf_z(X/G^z) \rightarrow \LLf_z(X/G)$ and $\LL_z^u(X/G^z) \rightarrow \LL_z^u(X/G)$ are equivalences.
\end{cor}
\begin{proof}
Loop spaces commute with fiber products, and $X/G^z = X/G \times_{BG} BG^z$.  The second claim follows since \'{e}tale maps induces equivalences on formal completions along isomorphic closed subschemes.  It is a straightforward verification that the map
$$U_{z, G^z} \times^{G^z} G = \{(huzh^{-1}, g) \mid g \in G, h \in G^z, u \in U_G \cap G^z\} \longrightarrow U_{z,G} = \{guzg^{-1} \mid u \in U_G \cap G^z\}$$
is an isomorphism.
\end{proof}

We define two competing notions of $z$-invariants.
\begin{defn}
In the set-up above, we define the \emph{derived $z$-invariants} $X^z$ to be the derived fiber product of the diagram
$$\begin{tikzcd}
X^z \arrow[r] \arrow[d] & X \arrow[d, "\Gamma_z"] \\
X \arrow[r, "\Delta"] & (X \times X)
\end{tikzcd}$$
where $\Gamma_z$ denotes the graph of the action by a closed point $z \in G$.
We define the \emph{classical $z$-invariants} $X^z$ to be $\pi_0(X^z)/G^z$, or equivalently, the above square considered as an underived fiber product.
\end{defn}

Given this notion of $z$-formal and $z$-unipotent loops, we are ready to define the localization map comparing $z$-formal and unipotent loops with the formal and unipotent loops of the classical $z$-fixed points.
\begin{defn}[Localization map and formal localization map]\label{MainDef}
We define the following maps realizing localizations from the most global to local.
\begin{itemize}
\item We define the \emph{global localization map} via the composition
$$\ell_z: \begin{tikzcd} \LL(\pi_0(X^z)/G^z) \arrow[r] & \LL(X/G^z) \arrow[r] & \LL(X/G)\end{tikzcd}$$
induced by the sequence of natural maps of quotient stacks $\pi_0(X^z)/G^z \rightarrow X/G^z \rightarrow X/G$ where the first map is the closed immeresion and the second map is the base change map along $BG^z \rightarrow BG$.  For $U/G \subset G/G$ an open subscheme, we by $\ell_{z, U}$ the restriction of $\ell_z$ to $U$.  The map lives over the natural map $\LL(BG^z) \rightarrow \LL(BG)$.
\item The \emph{unipotent localization map}  $\ell_z^u: \begin{tikzcd} \LL^u(\pi_0(X^z)/G^z) \arrow[r] & \LL^u_z(X/G)\end{tikzcd}$
is the base change of $\ell_z$ along
$$\begin{tikzcd}
\LL^u_z(BG^z) \arrow[d, "\simeq"] \arrow[r] & \LL(BG^z) \arrow[d]  \\
\LL^u_z(BG) \arrow[r] & \LL(BG)
\end{tikzcd}$$
where the isomorphism on the left arises via Corollary \ref{UniFormBGzBG}.  Note that the map lives over $\LL^u_z(BG)$.
\item The \emph{formal localization map} $\widehat{\ell}_z: \begin{tikzcd} \LLf_z(\pi_0(X^z)/G^z) \arrow[r] & \LLf_z(X/G)\end{tikzcd}$
is the base change of $\ell_z$ along
$$\begin{tikzcd}
\LLf_z(BG^z) \arrow[d, "\simeq"] \arrow[r] & \LL(BG^z) \arrow[d]  \\
\LLf_z(BG) \arrow[r] & \LL(BG)
\end{tikzcd}$$
where the isomorphism on the left arises via Corollary \ref{UniFormBGzBG}.  Note that the map lives over $\LLf_z(BG)$.
\end{itemize}
\end{defn}

\begin{rmk}
Applying the functor $- \times_{BG} \pt$, $\ell_z$ can be identified with the map induced on fiber products of the diagrams 
$$\begin{tikzcd} & G \times^{G^z} (\pi_0(X^z) \times G^z) \arrow[d] \\ G \times^{G^z} \pi_0(X^z) \arrow[r] & G \times^{G^z} (\pi_0(X^z) \times \pi_0(X^z)) \end{tikzcd} \longrightarrow \begin{tikzcd} & X \times G \arrow[d] \\ X \arrow[r] & X \times X \end{tikzcd}$$
where the top map sends $(h, x, g) \mapsto (h \cdot x, hgh^{-1})$.
\end{rmk}

We now investigate the map $\LL(\pi_0(X^z)/G^z)) \rightarrow \LL(X/G^z)$.  For ease of notation, we can replace $G^z$ with a reductive group $G$ in which $z$ is central.  Our goal is to find an open $G$-closed neighborhood $U/G \subset G/G$ on which the above map is an equivalence.  Let us first consider the case when $G = T$ is a torus.
\begin{lemma}[Finiteness of stabilizer subgroups]\label{TorusStab}
Let $T$ be a torus acting on a (quasicompact) variety $X$.  Only finitely many subgroups of $T$ may appear as stabilizers of this action.
\end{lemma}
\begin{proof}
We can work affine locally, since $X$ has a finite $T$-closed Zariski cover, and may also assume that $X$ is connected.  
If every point of $X$ has stabilizer of equal dimension to $T$, the possible stabilizer subgroups are in bijection with a subset of the  set of subgroups of the (finite) component group $T/T^\circ$.  If there is a point $x \in X$ whose stabilizer $T^x \subset T$ has strictly smaller dimension, then by the Luna slice theorem (note that the stabilizer $T^x$ is reductive since every subgroup is) there is a locally closed subvariety $V \subset X$ such that the map $a: T \times^{T^x} V \rightarrow X$ is \'{e}tale and dominant.  Any stabilizer of a point in the image of $a$ must be a subgroup of $T^x$, so the problem reduces to considering (a) the action of $T^x$ on $V$ along with (b) the action of $T$ on the compliment of $V \subset X$ (which is a closed subvariety, therefore affine, of strictly smaller dimension).  Note that in both cases, the dimension of either the variety or the group decreases, and that the claim is obviously true for zero-dimensional varieties and discrete (finite) groups, so the lemma follows by induction.
\end{proof}

\begin{cor}
In the above situation, let $z \in T$ and let $U \subset T$ be the open neighborhood of $z$ obtained by deleting the finitely many stabilizers which do not contain $z$.  Then, for $w \in U$, we have $X^w \subset X^z$.
\end{cor}

We bootstrap this reuslt to prove an analogous result for general reductive $G$.
\begin{prop}\label{NearbyPoints}
Let $G$ be a reductive group, $z \in G$ a central element and $X$ a scheme with a $G$-action.  Then, there is an open neighborhood $U$ of $z$ in $G$ that is closed under the adjoint action and such that $w \in U$ implies that $X^w \subset X^z$.
\end{prop}
\begin{proof}
We define $U$ as follows.  First, take a maximal torus $T \subset G$ containing $z$.  By Proposition \ref{TorusStab}, there is an open neighborhood $U'$ of $T$ with the given property.  We define
$$U = \{gtug^{-1} \mid t \in U', g \in G, u \in C_G(t)^{uni}\}$$
where $C_G(t)^{uni}$ are the unipotent elements in the centralizer of $t$.  The set $U$ is evidently $G$-closed.  By Jordan decomposition, $G - U = G \cdot (T - U')$, so $U$ is open since $T - U'$ is closed.  It remains to show that if $w \in U$, then $X^w \subset X^z$.

Every $w \in U$ has a Jordan decomposition $w = su$ where $s$ is semisimple and $u$ is unipotent; in particular, $s = gtg^{-1}$ for some $t \in U'$ and $g \in G$.  First, note that $X^s = g \cdot X^t \subset g \cdot X^z = X^z$ (since $X^z$ is $G$-closed for central $z$).  By the following lemma, it follows that $X^w \subset X^s \subset X^z$.
\end{proof}

\begin{lemma}
Let $w \in G$ be an element of a reductive group acting on a scheme $X$, with Jordan decomposition $w = su$ for semisimple $s$ and unipotent $u$.  Then, $X^w \subset X^s$.
\end{lemma}
\begin{proof}
If $u = e$ the claim is trivial, so suppose $u \ne e$.  Take $x \in X^w$.  There is a 1-parameter subgroup of $G$ containing $s$, and a choice of unipotent $u \ne 1$ uniquely defines an injective map of algebraic groups $\G_a \rightarrow G$, assembling into an injective map of group schemes $\G_m \times \G_a \rightarrow G$.  Then, $H = (\G_m \times \G_a) \cap G^x$ is a closed group scheme of $G$ which is at most two-dimensional and which contains $w = su$.  The connected component $H^\circ$ is either the trivial group, $\G_m \times \{0\}$, $\{1\} \times \G_a$, or the entire group $\G_m \times \G_a$.  If it is the trivial group then $H$ is discrete, therefore finite, but this is impossible since the projection to $\G_a$ is open and there are no finite group subschemes of $\G_a$.  For the same reason, $H^\circ$ cannot be $\G_m \times \{0\}$.  In the remaining two cases, $wH^\circ \subset H \subset G^x$, so $w \in G^x$ as desired.
\end{proof}

Given this, we are now ready to prove our main theorem.

\begin{thm}\label{MainThm}
Let $X$ be a smooth variety with an action of a reductive group $G$, and $z \in G$ semisimple.  There is an open substack $U/G \subset G/G$ containing $z$ such that the $S^1$-equivariant map
$$\begin{tikzcd} \ell_{z, U}: \LL(\pi_0(X^z)/G^z) \times_{G^z/G^z} (U \cap G^z)/G^z \arrow[r] &  \LL(X/G) \times_{G/G} U/G\end{tikzcd}$$
is \'{e}tale.  If $z \in G$ is central, then $\ell_{z, U}$ is an equivalence.  In particular, the $S^1$-equivariant maps on $z$-unipotent loops and $z$-formal loops are equivalences:
$$\begin{tikzcd} \ell^u_z: \LL^u_z(\pi_0(X^z)/G^z) \arrow[r, "\simeq"] & \LL^u_z(X/G), & &  \widehat{\ell}_z: \LLf_z(\pi_0(X^z)/G^z) \arrow[r, "\simeq"] & \LLf_z(X/G).\end{tikzcd}$$
\end{thm}

\begin{proof}
Assuming the first statement holds, then the statement on $z$-unipotent and formal loops follows since the global localization map is a composition
$\LL(\pi_0(X^z)/G^z) \rightarrow \LL(X/G^z) \rightarrow \LL(X/G).$
By the first statement of the theorem, since $z$ is central in $G^z$, there is an open substack $U$ of $G^z/G^z$ containing both $Z = G \cdot \{z\}$ and its saturation $U_z$ over which the first map is an equivalence.  By Proposition \ref{EtaleProp}, the second map is \'{e}tale on $U$, inducing an equivalence on formal completions along $Z$ and $U_z$ by the argument in Corollary \ref{UniFormBGzBG}.

We now prove the first statement.  Let $U'' \subset G^z$ be an open neighborhood of $z \in G^z$ from Proposition \ref{NearbyPoints}, i.e. such that $\pi_0(X^w) \subset \pi_0(X^z)$ for $w \in U''$.  Let $U' \subset G$ be an open neighborhood of $z$ on which $\phi$ is \'{e}tale, obtained via Proposition \ref{EtaleProp}; note that this implies that $\phi$ is open over $U'$.  Define $U = \phi(U'') \cap U' \subset G$.  By construction, the map $\LL(X/G^z) \rightarrow \LL(X/G)$ is \'{e}tale over $U$.  It remains to prove that the map is an equivalence when $z \in G$ is central; in this case $G^z = G$.

We proceed by base changing from $BG$ to $\Spec(k)$ (i.e. forgetting equivariance).  Take $Y = \pi_0(X^z)$ for shorthand, and recall that $U \subset G$ is an open subscheme on which $\pi_0(X^w) \subset \pi_0(X^z)$ for $w \in U$. This means we wish to show that the map of derived schemes
$$i: Y \times_{Y \times Y} (Y \times U) \rightarrow X \times_{X \times X} (X \times U)$$
is an equivalence.  Since $i$ is a closed embedding (and therefore affine), we view the map of derived schemes affine locally as a map of differential graded sheaves of algebras on $U \times G$, which we abusively denote by $\OO_{\LL(X/G)} \rightarrow \OO_{\LL(Y/G)} = i_* \OO_{\LL(Y/G)}$. Since these sheaves have coherent cohomology on $X \times G$ and closed points are dense in $X \times G$ ($X \times G$ is locally finite type over a field), it suffices to check the claim on local rings at closed points $(x, w) \in \pi_0(\LL(X/G)) \subset X \times G$ where $w \in U'' \subset G$ and $x \in \pi_0(X^w) \subset \pi_0(X^z)$ (via the defining property of $U''$).

Note that $Y = \pi_0(X^z)$ is smooth when $z \in G$ is semisimple in reductive $G$ by a standard argument\footnote{Choose a torus $T$ containing $z$ and apply the \'{e}tale slice theorem to a $T$-closed affine open cover.}.  Consequently, the diagonal maps are local complete intersections, and $Y = \pi_0(X^z) \subset X$ is also a local complete intersection.  The $z$-action on the cotangent space $T^*_x(X)$ is semisimple, and determines a splitting with identifications
$$T^*_x(X) = E_0 \oplus E_1, \;\;\;\;\; E_0 = \bigoplus_{\lambda \ne 1} \ker(z - \lambda) \simeq N^*_x(\pi_0(X^z)/X), \;\;\;\;\; E_1 = \ker(z - 1) \simeq T^*_x(\pi_0(X^z)).$$

Let $J$ denote the ideal such that $\OO_{Y, x} = \OO_{X, x}/J$.  We aim to compute
\begin{equation}\label{MainProofEqn}
\OO_{X, x} \otimes_{\OO_{X \times X, (x, x)}}^L \OO_{X \times G, (x, w)} \longrightarrow \OO_{Y, x} \otimes_{\OO_{Y \times Y, (x,x)}}^L \OO_{Y \times G, (x, w)} .\end{equation}
Let $v_1, \ldots, v_r$ be a basis of $E_1$ and $v_{r+1}, \ldots, v_n$ a basis of $E_0$.  By Nakayama's lemma we can lift this basis of the cotangent space to generators $x_1, \ldots, x_n \in \mf{m}_{X, x} \subset \OO_{X, x}$ such that $J = (x_1, \ldots, x_r)$.  Furthermore, again by Nakayama, $\{x_i \otimes 1 - 1 \otimes x_i \mid i = 1, \ldots, n\}$ form a regular sequence for the diagonal\footnote{Since $A$ and $A/J$ are Cohen-Macaulay, any minimal generating set is automatically regular.} 
$X \subset X \times X$ at $(x, x)$, and likewise the images of $\{x_i \otimes 1 - 1 \otimes x_i \mid i=r+1, \ldots, n\}$ form a regular sequence for the diagonal $Y \subset Y \times Y$ at $(x,x)$.

Taking semi-free Koszul resolutions of the diagonal, as well as $\OO_{Y \times G, (x,w)}$ as a $\OO_{X \times G, (x,w)}$-module, 
Equation \ref{MainProofEqn} can be rewritten
\begin{equation}\label{MainProofEqn2}
\begin{tikzcd}
\OO_{X \times G, (x,w)}[\epsilon_1, \ldots, \epsilon_n] \arrow[r] \arrow[dotted, bend right=10, rr] & \OO_{Y \times G, (x, w)}[\epsilon_{r+1}, \ldots, \epsilon_n] & \arrow[l] \OO_{X \times G, (x, w)}[\epsilon_1', \ldots, \epsilon_r', \epsilon_{r+1}, \ldots, \epsilon_n]
\end{tikzcd}
\end{equation}
where $|\epsilon_i| = -1$, and the internal differentials are defined by $d(\epsilon_i) = c(x_i) - x_i \otimes 1$ and $d(\epsilon_i') = x_i \otimes 1$.  The map on the right is a quasi-isomorphism.  To show that Equation \ref{MainProofEqn2} is a quasi-isomorphism, it suffices to produce the dotted map above making the diagram commute.  Explicitly, we wish to show that the derived equations imposed by $d(\epsilon_1), \ldots, d(\epsilon_r)$ and by $d(\epsilon_1'), \ldots, d(\epsilon_r')$ differ by an (invertible) change of variables.  That is, we wish to find an element of $GL_r(\OO_{X \times G, (x,w)})$ which transforms
$$\{x_1 \otimes 1, \ldots, x_r \otimes 1\} \;\;\;\;\;\; \text{into} \;\;\;\;\;\; \{c(x_1) - x_1 \otimes 1, \ldots, c(x_r) - x_r \otimes 1\}.$$

The element $z$ is central, so $w$ fixes $\pi_0(X^z)$, and in particular the coaction map $c$ preserves the ideal $J$.  Thus, we can write $c(x_i) = \sum_{i, j = 1}^r e_{ij} (x_j \otimes 1)$ for some $e_{ij} \in \OO_{X \times G, (x,w)}$.   Let $E = (e_{ij})$ denote the corresponding matrix; the matrix $E - I$ is invertible if its evaluation at $w \in G$ is invertible.  The matrix $E(w)$ is the action of $w$ on the conormal space $E_0 = N^*_x(\pi_0(X^z)/X)$; in particular, $(E - I)(w)$ is invertible if and only if $E(w)$ has no fixed vectors.  But $w$ cannot fix any vectors on a normal space of its fixed-point variety, and since $N_x(\pi_0(X^z/X) \subset N_x(\pi_0(X^w/X))$, the claim follows.
\end{proof}


\begin{cor}
Let $X, Y, Z$ be smooth varieties, with maps $f: X \rightarrow Z$ and $g: Y \rightarrow Z$.  Then
$$\LLf_z(\pi_0(X^z) \times_{\pi_0(Z^z)} \pi_0(Y^z)) \simeq \LLf_z(X \times_Z Y)$$
where all fiber products are derived.
\end{cor}
\begin{proof}
This follows immediately since loop spaces commute with fiber products.
\end{proof}

\begin{rmk}\label{RmkTorus}
Note that in the case $G = T$ is a torus, every element of $T$ is central, so $\ell_{z, U}$ is an equivalence.  Furthermore, Proposition \ref{TorusStab} gives an explicit description of the open set $U$, which is maximal, on which $\ell_z$ is an equivalence.  
\end{rmk}

\subsection{Equivariant localization for Hochschild and cyclic homology}

First, let us show that the completion over a closed point of the affinization $[z] \in G//G$ is the same as taking $z$-unipotent loops.
\begin{lemma}\label{UniB}
Let $a: G/G \rightarrow G//G$ be the affinization, and let $a(z) = [z] \in G//G$.  The map above induces an isomorphism on completions
$$\LL^u_z(BG) \simeq \widehat{a^{-1}([z])}/G \subset \LL(BG).$$
In particular, the map $\LL^u_z(BG) \rightarrow \LL(BG)$ factors isomorphically through $\widehat{a^{-1}([z])}/G$.
\end{lemma}
\begin{proof}
This follows from Proposition \ref{UnipotentLoopsQuotient} and the observation that the classical reduced fiber over $[z] \in G$ in $G/G$ is isomorphic to $U_{G^z}/G^z$, where $U_{G^z}$ is the unipotent cone of $G^z$.  That is, if $u \in G^z$ is unipotent, then $uz = zu \in G$ has the same eigenvalues as $z$, so $a(z) = a(zu)$, and by Jordan composition any $y \in \mu^{-1}([z])$ can be written uniquely in this way.  Furthermore, $G^{uz} \subset G^z$; letting $U$ be the unipotent elements of $G^z$, it follows from Proposition \ref{EtaleProp} that the map $G \times^{G^z} U \rightarrow G$ is a closed embedding with image $\mu^{-1}([z])$.
\end{proof}

As an application of the above geometric incarnations of equivariant localization, we obtain the following equivariant localization results in Hochschild and cyclic homology.  

\begin{defn}[Completion and localization of Hochschild homology]
Let $G$ be a reductive group.  Note that $HH(\Perf(X/G))$ is naturally a $HH(\Perf(BG)) = k[G]^G$-module.  Let $z \in G$ be a reductive element representing a closed point of $\Spec(k[G]^G) = G//G$.  We denote by $HH(\Perf(X/G))_z$ the localization at the maximal defining $z \in G//G$ and $HH(\Perf(X/G))_{\widehat{z}}$ the completion at $z$.  We define the analgous notions for $HN, HC$ and $HP$, taking care to completed with respect to the ideals $I_z[[u]]$ and $I_z((u))$ where $I_z \subset k[G]^G$ is the ideal defining $z \in G//G$.
\end{defn}

The following is an immediate corollary of the the localization theorem for unipotent loops.
\begin{thm}[Equivariant localization for Hochschild homology]\label{ThmHH}
Let $X$ be a smooth scheme, $G$ a reductive group, and $z \in G$ semisimple.  Then the natural map on completions induced by pullback is an equivalence:
$$\begin{tikzcd} HH(\Perf(X/G))_{\widehat{z}} \arrow[r, "\simeq"] & HH(\Perf(\pi_0(X^z)/G^z))_{\widehat{z}} = \OO(\LL^u_z(\pi_0(X^z)/G^z)).\end{tikzcd}$$
If $z \in G$ is central, the map on localizations is an equivalence:
$$\begin{tikzcd} HH(\Perf(X/G))_{z} \arrow[r, "\simeq"] & HH(\Perf(\pi_0(X^z)/G))_{z} = \OO(\LL(\pi_0(X^z)/G)) \otimes_{\OO_{G//G}} \OO_{G//G, [z]}.\end{tikzcd}$$
\end{thm}
\begin{proof}
Via Lemma \ref{UniB}, we have a natural identification $HH(\Perf(\pi_0(X^z)/G^z))_{\widehat{z}} \simeq \OO(\LL^u_z(X/G))$.  The first claim then follows directly from Theorem \ref{MainThm}, and base change for derived completions along closed embeddings.  The second claim for localizations follows by Theorem \ref{MainThm}, since in that case the map on Hochschild homology is an equivalence for an open set containing $z$.
\end{proof}

\begin{rmk}
For cyclic homology, the situation is a little more delicate.  The formation of cyclic homology $HC$ involves a filtered colimit, the formation of negative cyclic homology $HN$ involves a cofiltered limit, and the formation of periodic cyclic homology involves both.  On the other hand, localization commutes with colimits and finite limits but not cofiltered limits, and completion commutes with limits but not colimits.
\end{rmk}

We first need to introduce a few technical notions regarding mixed complexes, with the aim of proving that in our situation formal completions commute with the Tate construction on Hochschild homology.  Analogous results and arguments can be found in \cite{Ka:CP}.
\begin{defn}
Let $(V, d, \epsilon)$ be a mixed complex.  We define
$$V^{\prod\Tate} = (\prod_k V u^k, d + u\epsilon) \;\;\;\;\;\;\;\;\;\; V^{\oplus\Tate} = (\bigoplus_k V u^k, d + u\epsilon) = (V[u^{-1}, u], d + u\epsilon)$$
where $|u| = 2$.  There are natural maps
$$V^{\oplus\Tate} \rightarrow V^{\Tate} \rightarrow V^{\prod\Tate}.$$
\end{defn}

\begin{rmk}
Lemma 2.6 of \cite{Ka:CP} shows that the Tate construction preserves quasi-isomorphisms, essentially because it is computed via the right spectral sequence.  On the other hand, the other variants above do not preserve quasi-isomorphisms.  In particular, they are not well-behaved in the derived category.
\end{rmk}

\begin{defn}
We say a complex $V$ is \emph{cohomologically bounded below} (respectively, \emph{above}) if $H^i(V) = 0$ for all sufficiently small (respectively, large) $i$.  We say $V$ is \emph{strictly bounded below} (respectively, \emph{above})  if $V$ if $V^i = 0$ for $i$ sufficiently small (respectively, large).
\end{defn}

\begin{lemma}\label{TateProd}
Let $(V, d, \epsilon)$ be a mixed complex.  If $V$ is strictly bounded below, then $V^{\oplus\Tate} \rightarrow V^{\Tate}$ is an isomorphism.  If $V$ is strictly bounded above, then $ V^{\Tate} \rightarrow V^{\prod\Tate}$ is an isomorphism.
\end{lemma}
\begin{proof}
The proof of the first statement appears as Corollary 2.7 in \cite{Ka:CP}, which we repeat for convenience.  The chain complex $V^{\oplus \Tate}$ is, in the $n$th cohomological degree, the vector space $\bigoplus_k u^k V^{n-2k}$.  whereas $V^{\Tate}$ is in the $n$th cohomological degree the vector space $(\bigoplus_{k \geq 0} u^k V^{n-2k}) \times (\prod_{k < 0} u^k V^{n-2k}$.  Since $V^{n-2k} = 0$ for large $k$, the product is finite and therefore a direct sum, so the map is an isomorphism.  A similar argument proves the second statement. 
\end{proof}

\begin{lemma}\label{TateLimitColimit}
Let $(V_\alpha, d_\alpha, \epsilon_\alpha)$ be a degree-wise Mittag-Leffler sequential diagram of mixed complexes such that the $V_\alpha$ are uniformly cohomologically bounded above (i.e. right $t$-bounded).  Then, the functors $(-)_{S^1}$ and $(-)^{\Tate}$ commute with limits, i.e.
$$(\lim_\alpha V_\alpha)^{\Tate} \simeq \lim_\alpha V_\alpha^{\Tate}.$$
\end{lemma}
\begin{proof}
First, since the limit is degree-wise Mittag-Leffler, $R\lim = \lim$.  Note that since the $V_\alpha$ are uniformly cohomologically bounded above, so is their limit.  Since we are only interested in computing the usual Tate construction which respects quasi-isomorphisms, we can replace each of the $V_\alpha$ and $\lim V_\alpha$ with their (strictly bounded above) truncations; the resulting complex is still Mittag-Leffler.  Now, note that the $\prod\Tate$ construction commutes with limits. 
\end{proof}

\begin{thm}[Equivariant localization for cyclic homology]\label{ThmHC}
Let $X$ be a smooth scheme, $G$ a reductive group, and $z \in G$ semisimple.  Then the following maps on completions induced by pullback are equivalences
$$\begin{tikzcd} HN(\Perf(X/G))_{\widehat{z}} \arrow[r, "\simeq"] & HN(\Perf(\pi_0(X^z)/G^z))_{\widehat{z}} = \OO(\LL^u_z(\pi_0(X^z)/G^z))^{S^1}, \end{tikzcd}$$
$$\begin{tikzcd} HC(\Perf(X/G))_{\widehat{z}} \arrow[r, "\simeq"] & HC(\Perf(\pi_0(X^z)/G^z))_{\widehat{z}}= \OO(\LL^u_z(\pi_0(X^z)/G^z))_{S^1}, \end{tikzcd}$$
$$\begin{tikzcd} HP(\Perf(X/G))_{\widehat{z}} \arrow[r, "\simeq"] & HP(\Perf(\pi_0(X^z)/G^z))_{\widehat{z}}= \OO(\LL^u_z(\pi_0(X^z)/G^z))^{\Tate}. \end{tikzcd}$$
\end{thm}
\begin{proof}
Via Lemma \ref{UniB}, we have a natural identification $HH(\Perf(\pi_0(X^z)/G^z))_{\widehat{z}} \simeq \OO(\LL^u_z(X/G))$.  If we can show that the limit diagram in the formation of the derived completion satisfies the conditions of Lemma \ref{TateLimitColimit}, then the claims would follow via Theorem \ref{ThmHH}.

Take $A \in \Alg(\QCoh(BG))$ to be as in Proposition \ref{GeneratorQuotientStack}.  First, note that $HH(\Perf(X/G)) \simeq C_\bullet(A, G)$ is cohomologically bounded above since $A$ is cohomologically bounded.  By Remark \ref{EquivariantCyclicProjective}, its terms are projective, so the terms in the limit computing its derived completion with respect to an ideal of $k[G]^G$ are classical quotients.  Therefore, the limit diagram satisfies the conditions of Lemma \ref{TateLimitColimit}.
\end{proof}

\begin{rmk}
While the statement in the above theorem for negative cyclic homology was more or less automatic, the statements for cyclic homology and periodic cyclic homology are strongly dependent on the cohomological right-boundedness of Hochschild homology, which we expect to fail for $\Coh(X/G)$ where $X$ is singular.
\end{rmk}

\section{An Atiyah-Segal completion theorem for periodic cyclic homology}\label{HPSection}

\subsection{Twisted circle actions on loop spaces}\label{TwistedActionSection}

\begin{defn}
Let $G$ be an affine algebraic group over $k$.  A \emph{$G$-action} on a prestack $X$ is defined to be a prestack $Y$ over $BG$, along with an identification $Y \times_{BG} \pt \simeq X$.  We often abuse notation and write $Y = X/G$ and understand that the identification implicitly.  Furthermore, if $Z \subset G$ is a closed normal subgroup, a \emph{trivialization} of the action of $Z$ on $X$ is defined to be a $G' := G/Z$-action on $X$ along with an identification $X/G' \times_{BG'} BG \simeq X/G$.  If $z \in G$ is central, then it generates a normal closed subgroup $Z$; we sometimes write \emph{$z$-trivialization} to mean a $Z$-trivialization.

Furthemore, note that any $X$ equipped with a $z$-trivialization is equipped with a \emph{shift map} $\mu_z: X/G \rightarrow X/G$ defined as follows.  It is defined on $BG$ via the map of groups $G \rightarrow G$ taking $g \mapsto zg = gz$.  It is defined on $X/G$ by transporting the map $\mathrm{id}_{X/G'} \times_{\mathrm{id}_{BG'}} \mu_z$. across the trivialization $X/G \simeq X/G' \times_{BG'} BG$.  It is evidently an equivalence with inverse $\mu_{z^{-1}}$.
\end{defn}

\begin{exmp}
If $X$ is a classical scheme with a $G$-action, on which a normal subgroup $Z \subset G$ acts trivially, then there is a canonical trivialization of the action of $Z$ on $X$.
\end{exmp}

\begin{defn}[Twisted $S^1$-actions] \label{CircleActDefn}
Let $G$ a linear algebraic group over $k$ and $X$ a scheme over $k$ with an action of $G$.  We define the following various circle actions on $X/G$.  Let $z \in G$ be central.
\begin{itemize}
\item The \emph{loop rotation $S^1$-action}, denoted $\rho$, on the loop space $\LL(X/G)$.
\item The \emph{$z$-twisting $S^1$-action}, denoted $\tau_z$, on $BG$ is induced via the map of groups $\Z \times G \rightarrow G$ defined by $(n, g) \mapsto z^n g = g z^n$.  More generally, if $X$ is equipped with a $z$-trivialization, then the corresponding \emph{$z$-twisting $S^1$ action} on $X/G$ is defined via the identification $X/G \simeq X/G' \times_{BG'} BG$, where $S^1$ acts on $X/G'$ and $BG'$ trivially, but on $BG$ via the $z$-twisting described above.  It is evident that the maps defining the fiber product diagram are $S^1$-equivariant.

\item The \emph{$z$-twisted rotation}, denoted $\rho_z$, on $\LL(X/G)$ is the diagonal of the $S^1 \times S^1$-action $\rho \times \LL \tau_z$.  Note that it makes sense to talk about the diagonal since $\rho$ and $\LL \tau_z$ commute: $\rho$ commutes with any group action of the form $\LL \gamma$, where $\gamma$ is an $S^1$-action on $X/G$.
\end{itemize}
Furthermore, for $w \in G$ central, the map $\mu_z: X/G \rightarrow X/G$ is $S^1$-equivariant with respect to $\tau_{zw}$ on the left and $\tau_{w}$ on the right.  In particular, $\LL(\mu_z): \LL(X/G) \rightarrow \LL(X/G)$ is equivariant with respect to $\rho_{zw}$ on the left and $\rho_{z}$ on the right.
\end{defn}


\begin{rmk}\label{RmkTwisting}
Using Proposition \ref{LoopQuotient}, we see that the shift map
$\LL(\mu_z): \LL(X/G) \rightarrow \LL(X/G)$
has an explicit realization as the multiplication by $z$ map $G/G \rightarrow G/G$ on the upper-right term in the fiber product
	$$\begin{tikzcd}
	\LL(X/G) \times_{BG} \pt \arrow[r] \arrow[d] & (X \times G)/G \arrow[d, "a \times p"] \\
	X/G \arrow[r, "\Delta"] & (X \times X)/G.
	\end{tikzcd}$$
%
%
%
\end{rmk}

\begin{exmp}
On $\Perf(G/G)$, the rotation $\rho$ acts on fibers over $g \in G$ by $g$; the $z$-twisting $\LL \tau_z$ acts on fibers over any $g \in G$ by $z$, and the twisted rotation $\rho_z$ acts on fibers over $g \in G$ by $gz = zg$.
\end{exmp}

\begin{exmp}\label{BTExmp}
Let $G = T$ be a torus (in particular, every $t \in T$ is central).  We can explicitly describe the $S^1$-actions on linear categories
$$\Perf(\LL(BT)) = \Perf(T \times BT) \simeq \bigoplus_{\lambda \in \Lambda} \Perf(T)$$
where $\Lambda$ is the character lattice of $T$.  Let $\{z^\lambda \mid \lambda \in \Lambda\}$ denote the natural basis of monomials for $k[T]$.  The rotation $S^1$-action $\rho$ acts on the $\lambda$-summand by $z^\lambda$.  The $t$-twisting action $\tau_t$ on $BT$ acts on the $\lambda$-summand by the scalar $z^\lambda(t)$.  The $t$-twisted rotation $\rho_t$ acts on the $\lambda$-summand by $z^{\lambda}(t)z^\lambda$.

Let us take the $t$-twisted rotation $\rho_t$.  We have, via the category\footnote{For $M$ a scheme and $f: M \rightarrow \mathbb{G}_m$, the category $\PreMF(M, f)$ is the category $\Perf(M \times_{\mathbb{G}_m} \{1\})$ with an extra $k[u]]$-linear structure acting by cohomological operators.} $\PreMF$ in \cite{Pr:MF},
$$\Perf(\LL(BT))^{S^1} = \bigoplus_{\lambda \in \Lambda} \PreMF(T, 1 - z^\lambda(t) \cdot  z^{\lambda}),$$
$$\Perf(\LLf(BT))^{S^1} = \bigoplus_{\substack{\lambda \in \Lambda \\ \lambda(t) = 1}} \PreMF(\widehat{T}, 1 - z^\lambda).$$
For the second identity, the zeros of $1 - z^\lambda(t) z^\lambda$ meet the constant loops if and only if $z^\lambda(t) = 1$.
After passing to the Tate category under the rotation action, we note that the zero locus of $1 - z^\lambda$ is smooth (it is a subgroup of $T$) of codimension 1 unless $\lambda = 0$ (in which case it has codimension zero and must be derived).  Therefore,
$$\Perf(\LL(BT))^{\Tate} = \bigoplus_{\lambda \in \Lambda} \mathrm{MF}(T, 1 - z^\lambda(t) \cdot z^{\lambda}) = \Perf(T) \otimes_k k(u)),$$
$$\Perf(\LLf(BT))^{\Tate} = \bigoplus_{\substack{\lambda \in \Lambda \\ \lambda(t) = 1}} \mathrm{MF}(\widehat{T}, 1 - z^\lambda) = \Perf(\widehat{T}) \otimes_k k((u)).$$
Note that the Tate categories do not depend on the twisting at all (but the $S^1$-invariant categories do).
\end{exmp}

\begin{prop}\label{CircleActionsSame}
Suppose $z \in G$ is a central element of a reductive group and acts on a quasiprojective scheme $X$ trivially.  Then, there is a $k[[u]]$-linear equivalence
$$\LL(\mu_z)^*: \begin{tikzcd} \OO(\LL(X/G))^{S^1, \rho} \arrow[r, "\simeq"] & \OO(\LL(X/G))^{S^1, \rho_z}.\end{tikzcd}$$
The same holds for formal and unipotent loop spaces.
\end{prop}
\begin{proof}
Since $X$ is quasiprojective, it has a compact $G$-equivariant generator $\mathcal{E}$ of $\QCoh(X)$.  Let $A = R\Hom_X(\mathcal{E}, \mathcal{E})$, so that $\QCoh(X) \simeq A\dmod_{\QCoh(BG)}$.  Let $c: A \rightarrow A \otimes k[G]$ be the coaction; note that the $z$-twisted rotation is given by twisting the coaction by $z$, i.e. $c_z(a)(-, g) := c(a)(-, gz) = c(a)(-, zg)$.  Since $z$ acts on $X$ trivially, this is equal to the usual coaction $c$, inducing an equivalence of $S^1$-invariants under the untwisted and $z$-twisted rotations. 
\end{proof}

%

\subsection{Tate-equivariant functions on formal loop spaces compute analytic de Rham cohomology}

We now set out to prove the equivariant localization theorem for periodic cyclic homology.  We first introduce some technical notions needed to phrase the result in the 2-periodic setting.  Recall the following notions for vector spaces (not chain complexes) from \cite{Be:TA}.
\begin{defn}
A \emph{linear topological vector space} is a a vector space $V$ which admits a topology for which the vector space operations are continuous, and such that there is a system of neighborhoods at $0$ consisting of subspaces.  In this case, the topology is generated by this system at $0$ and translations under addition.  The \emph{completion} $\widehat{V}$ of $V$ is the limit over the system of neighborhoods $U_\alpha$:
$$\widehat{V} := \lim_{0 \in U_\alpha} V/U_\alpha.$$
We say the topology is \emph{complete} if the natural map $V \rightarrow \widehat{V}$ is an isomorphism.  Let $V_1, V_2$ be linear topological vector spaces.  
We define a linear topological vector space, the \emph{$!$-tensor product} $V_1 \otimes^! V_2$, via the naive tensor product on underlying vector spaces equipped with topology by the basis consisting of open sets of the form $U_1 \otimes V_2 + V_1 \otimes U_2$, where $U_1 \subset V_1$ and $U_2 \subset V_2$ are opens.  We define $V_1 \,\widehat{\otimes^!}\, V_2$ to be $V_1 \otimes^! V_2$ completed with respect to this topology.

These notions generalize immediately to chain complexes, where we replace the notion of subspace with subcomplex.  In this case, the complexes term-wise satisfy the Mittag-Leffler condition and therefore $\lim^1 = 0$, so $R\lim = \lim$.  Note that as in \cite{Ka:CP}
$$V_1 \,\widehat{\otimes^!}\,V_2 = \lim_{V_1} \lim_{V_2} V_1 \otimes V_2.$$
\end{defn}

\begin{rmk}
It is unclear to us how the notion of a topological chain complex should interact with quasi-isomorphisms.
\end{rmk}


%

We now review the constructions and results of \cite{Bh:DDR} and \cite{Ha:DR}.
\begin{defn}
Let $X$ be a finite-type derived stack with affine diagonal over $k$, and $\mathbb{L}_X$ its cotangent complex.  The \emph{derived de Rham complex} $\dR_X$ is the sum-totalization of the complex  $(\bigoplus_{n \geq 0} \bigwedge^n \mathbb{L}_X[-n], d_{dR})$, which comes equipped with a \emph{Hodge filtration} $F^k \dR_X = (\bigoplus_{n \geq k} \bigwedge^n \mathbb{L}_X[-n], d_{dR})$.  The \emph{Hodge-completed derived de Rham complex} $\hcdR_X$ is the completion of $\dR_X$ with respect to the Hodge filtration (see Construction 4.1 in \cite{Bh:DDR}).  We will denote the derived global sections of this complex to be the complex of \emph{derived de Rham cochains} $C^\bullet_{dR}(X; k)$.
\end{defn}

We have two competing notions of Hodge filtrations in the negative cyclic homology of $\Perf(X/G)$.  Both play an essential role.
\begin{defn}
Let $(V, d, \epsilon)$ be a mixed complex.  The \emph{noncommutative Hodge filtration} on $V^{S^1}$ and $V^{\Tate}$ is the decreasing filtration defined by the subspaces $u^k V^{S^1}$.  
\end{defn}

\begin{defn}
The odd tangent complex $\mathbb{T}_X[-1]$ is affine over $X$, so we can consider $\OO_{\mathbb{T}_X[-1]}$ as an algebra in the category $\QCoh(X)$. Furthermore, the ideal sheaf for the zero section defines an exhaustive decreasing filtration on $\OO_{\mathbb{T}_X[-1]}$, which induces an exhaustive decreasing filtration on $\OO(\mathbb{T}_X[-1])^{S^1}$ whose completion is $\OO(\widehat{\mathbb{T}}_X[-1])^{S^1}$.  We call this the \emph{geometric Hodge filtration}.
\end{defn}

\begin{prop}\label{StackCoh}
Let $X$ be a geometric stack with a smooth cover by a variety.  There is a natural quasi-isomorphism
$$\begin{tikzcd}\OO(\LLf(X))^{\Tate} \arrow[r, "\simeq"] & C^\bullet_{dR}(X; k) \,\widehat{\otimes}^!\, k((u))\end{tikzcd}$$
where we consider the de Rham complex $C^\bullet_{dR}(X; k)$ as a topological chain complex with respect to the derived Hodge filtration, and $k((u))$ with respect to the noncommutative Hodge filtration.
\end{prop}
\begin{proof}
By Theorem 6.9 of \cite{BZN:LC}, the exponential map $\widehat{\mathbb{T}}_{X}[-1] \rightarrow \LLf(X)$ is an filtration-preserving isomorphism, so we can compute $\OO(\widehat{\mathbb{T}}_X[-1])^{\Tate}$ instead.  Note that $(-)^{S^1}$ commutes with totalization, so we first compute
$$\OO(\widehat{\mathbb{T}}_X[-1])^{S^1} = \lim_{n \rightarrow \infty} \lim_{m \rightarrow \infty} \left( \bigoplus_{\substack{0 \leq i \leq n \\ 0 \leq j \leq m}} R\Gamma(X, \Omega^i_X[i])u^j,\;\; u \cdot d_{dR}\right).$$
Note that the filtration defined by the limit parameter $n$ is the geometric Hodge filtration, and the filtration defined by the limit parameter $m$ is the non-commutative Hodge filtration.  As we take the limit with respect to both, it amounts to computing the direct sum complex $(\bigoplus_{\substack{0 \leq i,j}} R\Gamma(X, \Omega^i_X[i])u^j,\;\; u \cdot d_{dR})$
with respect to the opens
$$U_{nm} = \left( \bigoplus_{\substack{i \geq n \text{ or}\\ j \geq m}} R\Gamma(X, \Omega^i_X[i])u^j,\;\; u \cdot d_{dR}\right).$$

There is a subcomplex (in fact, a direct summand) of the direct sum complex defined by taking the summands where $j \geq i$: 
$$V = \left( \bigoplus_{\substack{0 \leq i \leq j}} R\Gamma(X, \Omega^i_X[i])u^j,\;\; u \cdot d_{dR}\right).$$
It is a subcomplex since the de Rham differential takes the $(i, j)$-summand to the $(i+1, j+1)$-summand.  Its quotient
$$V' = \left( \bigoplus_{\substack{0 \leq j < i}} R\Gamma(X, \Omega^i_X[i])u^j,\;\; u \cdot d_{dR}\right)$$
is $u$-torsion.  Note that $k((u))$ is flat as a (dg) $k[[u]]$-module, so $- \otimes_{k[[u]]} k((u))$ kills $u$-torsion modules, and in particular we have an equivalence induced by the inclusion on completions
$$\widehat{V} \otimes_{k[[u]]} k((u)) \simeq \OO(\widehat{\mathbb{T}}_X[-1])^{\Tate}.$$


Consider the alternative topology on $V$ defined by opens
$$W_{km} = \left( \bigoplus_{\substack{i-j \geq k \text{ or}\\ j \geq m}} R\Gamma(X, \Omega^i_X[i])u^j,\;\; u \cdot d_{dR}\right).$$
The parameter $k$ in this topology defines the derived Hodge filtration in the derived de Rham complex, and the parameter $m$ in this filtration defines a $u$-adic filtration.  In particular, the completion of $V$ with respect to the topology defined by the $W_{km}$ is $\Gamma(X, \widehat{\dR}_X) \,\widehat{\otimes}^!\,k[[u]]$.

We claim these two topologies defined by $U_{nm}$ and $V_{km}$ are equivalent.  Indeed, this is an easy verification as all indices are bounded below.  Explicitly, $W_{nm} \subset U_{nm}$ and $U_{k+m,m} \subset W_{km}$.  In particular, the completion with respect to this topology is the limit under the usual Hodge filtration on each summand of $V$ defined by letting $i - j$ be constant, and the limit under the noncommutative Hodge filtration.  Thus, via the universal property of the limit, we have a canonical equivalence
$$\widehat{V} \simeq \Gamma(X, \widehat{\dR}_X) \,\widehat{\otimes}^!\,k[[u]]$$
and in particular, 
$$\OO(\LLf(X))^{\Tate} \simeq V \otimes_{k[[u]]} k((u)) \simeq \Gamma(X, \widehat{\dR}_X) \,\widehat{\otimes}^!\, k((u)) \simeq C^\bullet_{dR}(X; k) \,\widehat{\otimes^!}\, k((u)).$$
\end{proof}

\begin{rmk}
The above proposition is false if we do not consider the topologies.  For example, take $X = B\G_m$.  Then, we have $\OO(\LLf(B\G_m)) = k[[[t]]$ where $|t| = 0$, and in particular, $H^0(\OO(\LLf(B\G_m))^{\Tate}) = k[[t]]$.  On the other hand, $H^\bullet(B\G_m;k) \simeq k[s] = k[[s]]$ where $|s| = 2$, so $H^0(H^\bullet(B\G_m;k)((u))) = k[su^{-1}]$.
\end{rmk}

Finally, we discuss results relating derived de Rham cohomology to analytic (Betti) cohomology when $k = \C$.  It is well-known by experts and is essentially a simple corollary of results in \cite{Bh:DDR}, \cite{Ha:DR} and \cite{BZN:LC}.  A general discussion can also be found in the introduction of \cite{Ka:CP}.  
We first define a few intermediate chain complexes.
\begin{defn}
For a \emph{choice} of embedding of $i: X \rightarrow M$ for smooth $M$, the \emph{Hartshorne algebraic de Rham complex} $\Omega^H_X$ is defined by
$$\Omega^H_X := \widehat{i^{-1}\OO_M} \otimes_{i^{-1}\OO_M} i^{-1}\Omega^\bullet_Y.$$

We define the \emph{term-wise Hodge-completed derived de Rham complex} associated to a groupoid presentation $\begin{tikzcd} U_1 \arrow[r, shift left] \arrow[r, shift right] &U_0 \end{tikzcd}$ of a stack as above to be the totalization of the double complex $\hcdR_{U_j}$.  Since for schemes $U_i$, the de Rham cohomology computes Betti cohomology, Lemma 32 of \cite{Beh:CohStacks} implies that the cohomology of this complex is independent of choice of cover.
\end{defn}

Corollary 4.27 in \cite{Bh:DDR} and Theorem 1.1 in Chapter IV of \cite{Ha:DR} can be summarized in the following statement, identifying the derived de Rham cohomology of a possibly singular scheme $X$ over $\C$ with its Betti cohomology. 
For details, see Construcion 4.25 in \cite{Bh:DDR}.
\begin{thm}[Bhatt, Hartshorne]\label{BhattThm}
Let $X$ be a finite type scheme over $k=\C$.  There is a natural map of sheaves of dg $k$-algebras $\hcdR_X \rightarrow \Omega^H_X$ on $X$ 
which is a quasi-isomorphism.  Letting $j: X^{an} \rightarrow X$ be the analytification map, and $\Omega^\bullet_{X^{an}}$ the analytic de Rham complex on $X^{an}$, the map of dg $k$-algebras $j^{-1} \Omega^H_X \rightarrow \Omega^\bullet_{X^{an}}$
is a quasi-isomorphism.  All together, the map
$j^{-1} \hcdR_X \rightarrow \Omega^\bullet_{X^{an}}$
is a quasi-isomorphism, and $\Omega^\bullet_{X^{an}}$ resolves the constant sheaf $\C_{X^{an}}$.  Therefore, the hypercohomology of the derived de Rham complex computes the Betti cohomology of $X^{an}$.
\end{thm}

\begin{cor}
Let $X$ be a finite type stack over $k=\C$ presented as a groupoid $\begin{tikzcd} U_1 \arrow[r, shift left] \arrow[r, shift right] &U_0 \end{tikzcd}$.  Then, Betti cohomology of $X^{an}$ is computed by the term-wise Hodge-completed derived de Rham complex associated to $U_\bullet.$
\end{cor}

\begin{cor}
Let $X$ be a finite type derived stack over $k=\C$.  Then, there is a natural quasi-isomorphism
$$\begin{tikzcd} \OO(\LLf X)^{\Tate} \arrow[r, "\simeq"] & C^\bullet_{dR}(X^{an}; \C) \,\widehat{\otimes}_k \,k((u)).\end{tikzcd}$$
\end{cor}
\begin{proof}
By Proposition 6.4 of \cite{BZN:LC}, the Hodge-completed graded algebra of differential forms is a graded sheaf on the smooth site of $X$, i.e. the Hodge-completed derived de Rham complex for stacks is quasi-isomorphic to the term-wise Hodge-completed derived de Rham complex associated to an atlas $U \rightarrow X$ of the stack $X$.  This implies that the derived de Rham complex computes Betti cohomology and in particular, using Proposition \ref{StackCoh} there is an filtration-preserving equivalence in the derived category $C^\bullet_{dR}(X; \C) \simeq C^\bullet(X^{an}; \C)$.
\end{proof}

\subsection{Comparing global functions on unipotent and formal loop spaces}\label{ProGradedSection}

We prove a completion theorem for periodic cyclic homology, assuming the following theorem, which is proven as Theorem \ref{FormalUniProof}.
\begin{thm}\label{FormalUni}
Let $X$ be a quasicompact algebraic space with an action of an affine algebraic group $G$.  The map on functions induced by pullback is an equivalence:
$$\begin{tikzcd}\OO(\LL^u(X/G))^{\Tate} \arrow[r, "\simeq"] & \OO(\LLf(X/G))^{\Tate}\end{tikzcd}.$$
\end{thm}

\begin{thm}[Atiyah-Segal completion for periodic cyclic homology]\label{HPLoc}
Let $G$ be a reductive group acting on a smooth quasi-projective variety $X$.  The periodic cyclic homology $\HP(\Perf(X/G))$ is naturally a module over 
$HP(\Perf(BG)) = k[G//G]((u))$.  For a closed point $z \in G//G$, we have an identification of the formal completion at $z$ with a 2-periodicization of the singular cohomology of the fixed points
$$\begin{tikzcd}HP(\Perf(X/G))_{\widehat{z}} \arrow[r, "\simeq"] & C_{dR}^\bullet((X^z)^{an}/(G^z)^{an}; k) \,\widehat{\otimes}_k^!\, k((u))\end{tikzcd}$$
as a module over $HP(\Perf(BG))_{\widehat{z}} \simeq C_{dR}^\bullet(B(G^z)^{an}; k) \otimes_k^! k((u))$, contravariantly functorial with respect to $X$.
\end{thm}
\begin{proof}
By Theorem \ref{ThmHH} and Theorem \ref{MainThm},
$$HP(\Perf(X/G))_{\widehat{z}} \simeq \OO(\LL^u_z(X/G))^{\Tate} \simeq \OO(\LL^u(\pi_0(X^z)/G^z))^{\Tate, \rho_z}.$$
By Proposition \ref{CircleActionsSame} and Theorem \ref{FormalUni},
$$\OO(\LL^u(\pi_0(X^z)/G^z))^{\Tate, \rho_z} \simeq \OO(\LL^u(\pi_0(X^z)/G^z))^{\Tate} \simeq \OO(\LLf(\pi_0(X^z)/G^z))^{\Tate}.$$
Finally, by Theorem \ref{StackCoh}, we have
$$\OO(\LLf(\pi_0(X^z)/G^z))^{\Tate} \simeq C_{dR}^\bullet(\pi_0(X^z)^{an}/(G^z)^{an}; k)((u)).$$
\end{proof}

It remains to prove Theorem \ref{FormalUni}.  Central to our proof will be to use the fact that the map is a pro-graded isomorphism; the following lemma establishes a general situation when this is true.  Let us first clarify what we mean by a pro-graded isomorphism, and why this notion is necessary.

\begin{defn}
A \emph{pro-graded chain complex} $V$ is an object of $\Pro(\QCoh(B\G_m))$; that is, it is a filtered limit of graded chain complexes\footnote{By Proposition 1.1.3.6 of \cite{Lu:HA}, this category is stable.}.  Letting $\mathcal{L}_n$ denote the weight $n$ twisted one-dimensional $\G_m$-representation, the \emph{$n$th homogeneous part} functor is given by
$$(-)^{\text{wt}=n} := \mathrm{ev} \circ \Pro(\Gamma)(B\G_m, - \otimes \LL_{-n}): \Pro(\QCoh(B\G_m)) \rightarrow \cat{Vect}_k.$$
where the functor 
$$\Pro(\Gamma)(B\G_m, -):  \Pro(\QCoh(B\G_m)) \rightarrow \Pro(\cat{Vect}_k)$$ 
is the functor induced on pro-completions from $\Gamma(B\G_m, -): \QCoh(B\G_m) \rightarrow \QCoh(\pt)$ and the functor 
$$\mathrm{ev}: \Pro(\cat{Vect}_k) \rightarrow \cat{Vect}_k$$
 takes a limit diagram and evaluates it in $\cat{Vect}_k$ (which has all limits); it is right adjoint to the inclusion.   The \emph{underlying chain complex} is given by
$$\mathrm{ev} \circ \Pro(p^*): \Pro(\QCoh(B\G_m)) \rightarrow \cat{Vect}_k$$
where $p: \pt \rightarrow B\G_m$ is the usual atlas so that $p^*$ is the forgetful functor.  A map of graded chain complexes is a \emph{pro-graded isomorphism} if it is an isomorphism on $n$th graded parts for all $n$.
\end{defn}

\begin{rmk}
We require this formalism for the following reason. Let $p: \pt \rightarrow B\G_m$ be the standard atlas; the pullback (forgetful functor) $p^*$ does not commute with limits since the category $\QCoh(B\G_m)$ cannot differentiate between direct sums and direct products across different weights.  In particular, objects of $\QCoh(B\G_m)$ are $\mathbb{Z}$-graded chain complexes, which are equal to the direct sum of their homogeneous pieces\footnote{To see this, note that objects of $\QCoh(B\G_m)$ are chain complexes which are $\OO(\G_m)$-coalgebras, i.e. equipped with a map $V \rightarrow V \otimes_k k[z,z^{-1}]$.  In particular, tensors have finite rank, so any vector can only have finitely many homogeneous parts.}.
\end{rmk}

\begin{exmp}
For example, completing $\{0\} \subset \mathbb{A}^1/\G_m$ (under the usual scaling action),
$$\lim_{n, \QCoh(\pt)} k[x]/x^n = k[[x]] \;\;\;\;\;\;\; \lim_{n,\QCoh(B\G_m)} k[x]/x^n = k[x].$$  
\end{exmp}

\begin{rmk}
One way to remedy this is to keep track of the limit diagrams by working in the category $\Pro(\QCoh(B\G_m))$ and apply the evaluation functor in $\cat{Vect}_k = \QCoh(\pt)$ rather than $\QCoh(B\G_m)$.    
However, the category $\Pro(\QCoh(B\G_m))$ contains more information than we need: we do not wish to track the topologies on vector spaces as we only care about their completions.  Instead, we consider a smaller category: the category of $k^\Z$-modules
\end{rmk}

\begin{defn}
We define a functor (morally, some kind of Cartier duality) $D: \QCoh(B\G_m) \rightarrow \QCoh(\Z)$ as follows: it is a standard calculation that $\QCoh(B\G_m) \simeq \bigoplus_{n \in \Z} \QCoh(\pt)$.  For $V \in \QCoh(B\G_m)$, we denote by $V_n$ the summand corresponding to $n \in \Z$.  Define $D(\bigoplus V_n) = \bigoplus i_{n,*} V_n$ where $i_n: \Spec(k) \rightarrow \Z$ is the inclusion of the point $\{n\}$.

Letting $r: \mathbb{Z} \rightarrow \Spec(k^\Z)$ denote the affinization map, 
we define a functor $\Psi: \Pro(\QCoh(B\G_m)) \rightarrow \QCoh(\Spec k^\Z)) = k^\Z\dmod$ via the composition
$$\begin{tikzcd}
\Pro(\QCoh(B\G_m)) \arrow[r, "\Pro(D)"] \arrow[rrr, bend right=10, "\Psi"] & \Pro(\QCoh(\Z)) \arrow[r, "\Pro(r_*)"] & \Pro(\QCoh(\Spec(k^\Z))) \arrow[r, "\text{ev}"] & \QCoh(\Spec(k^\Z)). 
\end{tikzcd}$$
\end{defn}

\begin{defn}
For any $n \in \Z$, there is a natural map $\iota_n: \Spec k \rightarrow \Spec k^\Z$ defined by projection to the $n$th coordinate.  Denote by $k_n = \iota_{n,*} k \in \QCoh(\Spec k^\Z)$, and note that $k_n$ is projective since it is a summand of the free module decomposed by $k^\Z = k_n \oplus (\prod_{m \neq n} k_m)$.  We define the \emph{$n$-stalk} of a $k^\Z$-complex $M$ to be $\iota_n^* M = k_n \otimes M$ and the \emph{$n$-costalk} to be $\iota_n^! M =\Hom(k_n, M)$; since $k_n$ is projective the underived functor is the derived functor.
\end{defn}

\begin{rmk}
The notion of costalk and stalk are canonically equivalent.  It is a direct verification that in fact, $i^!(V) = i^*(V) = V_n$.  Let $p_n \in k^\Z$ denote the element with $0$ in the $n$th component and $1$ everywhere else.  Let $e_n$ denote the element with $1$ in the $n$th component and $0$ everywhere else.  Then, $p_n + e_n = 1$, so we find that $V = p_nV \oplus e_nV$, and $i^!(V) = \ker(p_n)$ while $i^*(V) = \mathrm{coker}(p_n)$.
\end{rmk}

Composing $\Psi$ with the global sections functors recovers the underlying vector space, and composing with the costalk at $n \in \Spec k^\Z$ recovers the $n$th homogeneous part.  In particular, if we are interested in studying the underlying vector space of $V \in \Pro(\QCoh(B\G_m))$ via its homogeneous components, it suffices to consider it as an object of $\QCoh(\Spec(k^\Z))$.
\begin{prop}
We have commutative diagrams of functors
$$\begin{tikzcd}
\Pro(\QCoh(B\G_m)) \arrow[dr, "{\mathrm{ev} \circ \Pro(p^*)}"'] \arrow[rr, "\Phi"] & & k^\Z\dmod \arrow[dl, "{\Gamma(\Spec k^\Z, -)}"] & \Pro(\QCoh(B\G_m)) \arrow[dr, "{{\mathrm{ev} \circ \Pro(\Gamma)(B\G_m, - \otimes \mathcal{L}_{-n})}}"'] \arrow[rr, "\Phi"] & & k^\Z\dmod \arrow[dl, "\iota_n^!"] \\
& \cat{Vect}_k &  & & \cat{Vect}_k & 
\end{tikzcd}$$
\end{prop}
\begin{proof}
We factor the first diagram as follows
$$\begin{tikzcd}
\Pro(\QCoh(B\G_m)) \arrow[dr, "\Pro(p^*)"'] \arrow[r, "\Pro(r_* \circ D)"] & \Pro(k^{\Z}\dmod)  \arrow[d, "{\Pro(\Gamma)}"] \arrow[r, "{\mathrm{ev}}"] & k^{\Z}\dmod  \arrow[d, "{\Gamma}"]\\
&  \Pro(k\dmod)  \arrow[r, "{\mathrm{ev}}"] & k\dmod
\end{tikzcd}$$
and the second in the analogous way.  So, the proposition follows from two claims: (1) that the diagrams above commute without the $\Pro$, i.e. $p^* = \Gamma(\Spec k^\Z, -) \circ r_* \circ D$ and $\Gamma(B\G_m, -) = \iota_0^! \circ r_* \circ D$, and (2) that the evaluation functor $\Pro(k^\Z\dmod) \rightarrow k^\Z\dmod$ commutes with global sections and taking costalks.

The first claim can be directly verified: it suffices to consider abelian categories since all functors are exact.  In particular, if $V \in \QCoh(B\G_m)^{\heartsuit}$, then it is a $k[z,z^{-1}]$-comodule, i.e. there is a map
$$V \rightarrow V \otimes k[z,z^{-1}] \simeq \bigoplus_n V_n z^n.$$
The functor $D$ takes $V$ to the complex on $\Z$ whose value on open affine $\{n\} \in \Z$ is $V_n$.  The functor $r_*$ takes $D(V)$ to $\bigoplus_n V_n$ where $k^\Z$ acts in the natural way.  Finally, we see that the global sections are exactly $p^* V = \bigoplus V_n$ and the costalk $i_n^!(\bigoplus_n V_n) = \Hom_{k^\Z}(k_n, \bigoplus_n V_n) = V_n$.

Since the evaluation functor is a right adjoint, we prove the second claim by showing that both the global sections functor and the costalks functor are right adjoints, and then using the general fact that right adjoints commute.  The global sections functor is right adjoint to the restrictions of scalars functor (it is also left adjoint to the ``corestriction of scalars'' functor $R\Hom_k(k^\Z, -)$).  The costalks functor is right adjoint to the pushforward (and in this case is equal to the stalks functor, which is a left adjoint).  
\end{proof}

\begin{rmk}
	In fact, since the global sections and (co)stalk functors are both left and right adjoints in the above situation, the above construction and proposition are valid for any iteration of taking $\Pro$ and $\Ind$ categories.
\end{rmk}

\begin{defn}
Let $V$ be a $k^{\Z}$-module.  The \emph{support} of $v \in V$ is the closed subscheme defined by the annihilator ideal of $v$.
\end{defn}

%

\begin{lemma}
A pro-graded isomorphism is injective.
\end{lemma}
\begin{proof}
This is the easy fact that if a map of sheaves on $\Spec(k^\Z)$ is zero on stalks at closed points, then it is zero, and the observation that on $\Spec k^\Z$, costalks and stalks coincide.
\end{proof}

\begin{lemma}\label{ProgradedFinite}
Suppose that $f: V \rightarrow W$ is a pro-graded isomorphism of pro-graded vector spaces such that 
either $V$ or $W$ are supported at finitely many weights.  Then $f$ is an isomorphism on underlying vector spaces.  More generally, let $A$ be a sheaf of algebras on $\Spec(k^\Z)$, and $f: V \rightarrow W$ a pro-graded isomorphism of sheaves of $A$-modules where $W$ is generated by elements supported at finitely many weights.  Then, $f$ is an isomorphism on underlying vector spaces.
\end{lemma}
\begin{proof}
For the first claim, the assumptions of the proposition imply that $V$ and $W$ have finite support, whose points consist entirely of closed points of $k^\Z$.  A map being a pro-graded isomorphism means that it is an isomorphism at stalks of closed points.

For the second more general claim, note that if $W$ is an $A$-module, and $w \in W$ is an element of finitely many weights, say $w = w_1 + \cdots + w_r$ where the $w_i$ are homogeneous of weight $c_i$, then $w_i \in A \cdot w$ since $e_{c_i} \cdot w = w_i$, where $e_{c_i} \in k^\Z$ is the characteristic function at $i \in \Z$.  In particular, $W$ having a set of generators supported at finitely many weights is equivalent to $W$ having a set of homogeneous generators.  Now, if $f: V \rightarrow W$ is a pro-graded isomorphism, then it is injective by the previous lemma.  For surjectivity, note that for a given homogeneous $w \in W$, since $f$ is a pro-graded isomorphism, we have a homogeneous $v \in V$ such that $f(v) = w$, and surjectivity follows since $W$ is generated by homogeneous elements.
\end{proof}

The following lemma allows us to reduce statements in the derived category to statements in the abelian category.
\begin{lemma}
A map $f: V \rightarrow W$ is a pro-graded quasi-isomorphism of pro-graded complexes if and only if each of the $H^i(f): H^i(V) \rightarrow H^i(W)$ are pro-graded isomorphisms of modules.
\end{lemma}
\begin{proof}
We can take a map $f: V \rightarrow W$ of complexes of $k^\Z$-modules.  It is easy to verify that taking $n$th homogeneous parts (i.e. talking stalks via localization) is exact, so that if $H^n(f)$ is a pro-graded isomorphism, it is an isomorphism of modules and therefore $f$ is a quasi-isomorphism.  Conversely, the global sections functor is clearly exact. 
\end{proof}

\begin{defn}
A \emph{pro-graded} dg-algebra is an object of $\Pro(\Alg(\QCoh(B\G_m)))$.  If $A$ is a pro-graded dg-algebra, then $\Spec(A)$ is naturally a \emph{dg-indscheme with a $\G_m$-action} in the sense of \cite{GR:DGI}.  We will use the word \emph{ind-stack} to mean a prestack which can be written as an inductive limit of closed embeddings of (derived) QCA stacks (in the sense of \cite{DG:QCA}); in practice we only need the case of a formal completion of a closed substack of a quotient stack.
\end{defn}

Recall the definition of a contracting $\G_m$-action in Definition \ref{DefnContracting}.
\begin{lemma}\label{CompleteWeights}
Let $A$ be a noetherian weight $\mathbb{Z}^{\leq 0}$ pro-graded connective dg-algebra, which is generated in negative weights over its weight $0$ part, and let $I = \pi_0(A^{\mathrm{wt}<0}) \subset \pi_0(A)$ be the classical augmentation ideal.  The derived completion $A \rightarrow \widehat{A_I}$ is a pro-graded quasi-isomorphism.  Globally, if $X$ is an ind-stack with a representable contracting $\G_m$-action with fixed point locus $Z \subset X$, then $\OO_X \rightarrow \OO_{\widehat{X_Z}}$ is a pro-graded quasi-isomorphism of quasicoherent sheaves on $X$.  In particular, $\OO(X) \rightarrow \OO(\widehat{X_Z})$ is a pro-graded isomorphism.
\end{lemma}
\begin{proof}
Choose generators $f_1, \ldots, f_r$ of $I$.  By Proposition \ref{DerivedCompletion} we can compute the homotopy limit $\widehat{A_I}$ via the limit
$$\widehat{A_I} = \lim_n  \left( A \otimes_{\pi_0(A)} K^\bullet_n \right)$$
where $K^\bullet_n$ is the Koszul complex for $f_1^n, \ldots, f_r^n \in \pi_0(A)$.  Since $f_1, \ldots, f_r$ are of strictly negative weight, for large $n$, $(K^\bullet_n)^{\text{wt} \geq -k} = (\pi_0(A))^{\text{wt} \geq -k}$ for any $k$.  Furthermore, homotopy limits can be computed in the derived category of $k$-complexes, and in particular we can compute the homotopy limit on each graded piece.  Thus, $(\widehat{A_I})^{\text{wt}=k}$ is a computed by a limit which stabilizes at $A^{\text{wt}=k}$, proving the claim.  For the global claim where $X$ is an ind-scheme, one can pass to an open affine $\G_m$-closed cover (which exists since $\G_m$ is a torus).  For the global claim where $X$ is an ind-stack, one can check the equivalence on a cover of $X$.
\end{proof}

Recall that by Remark 6.11 of \cite{BZN:LC} that there are embeddings $\LLf(X/G) \hookrightarrow \LL^u(X/G) \hookrightarrow \mathbb{T}_X[-1]$.  Thus, formal loops and unipotent loops inherit compatible $\G_m$-actions and their functions are $\mathbb{Z}^{\leq 0}$ pro-graded.    We have the following.
\begin{lemma}\label{GradedIso}
Let $X$ be a geometric stack.  The map induced by pullback
$$\begin{tikzcd} \OO(\LL^u X) \arrow[r, "\simeq"] & \OO(\LLf X)\end{tikzcd}$$
is a pro-graded (quasi-)isomorphism.
\end{lemma}
\begin{proof}
An argument is outlined in Corollary 2.7 of \cite{BZN:LR}; we will repeat it for convenience.  By Lemma \ref{LoopsGeometric}, $\LL X$ is geometric.  The formal loops $\LLf X$ are the completion of the unipotent loops $\LL^u X$ along constant loops, and the action is contracting by Lemma \ref{LemContracting}.  The statement follows by Lemma \ref{CompleteWeights}.  
\end{proof}

\begin{exmp}\label{CounterExmpTateUni}
The pro-graded isomorphism of Lemma \ref{GradedIso} may fail to be an isomorphism.  For example, take $X = G/U$ where $U$ is any unipotent subgroup of $G$; then we have that
$$\LL((G/U)/G) = \LL(BU) = U/U \rightarrow \LL(BG) = G/G$$
has image inside the unipotent cone of $G$.  In particular,
$$\LL^u(BU) = \LL(BU) = U/U \;\;\;\;\;\;\; \LLf(BU) = \widehat{\mf{u}}/U.$$ 
For example, if $U = B\G_a$, then the map is
$$\OO(\LL^u(B\G_a)) = \OO(\G_a \times B\G_a) = k[x, \eta] \rightarrow \OO(\LLf(B\G_a)) = \OO(\widehat{\G}_a \times B\G_a) = k[[x]][\eta]$$
is a pro-graded isomorphism but not an isomorphism, where $|x| = 0$ is a generator for $\OO(\G_a)$ and $|\eta| = 1$ is a generator for $\OO(B\G_a)$.
\end{exmp}

Using the fact that the map is a pro-graded isomorphism, we can show in the case of a unipotent group that the map on $\Tate$-equivariant functions is an isomorphism by a finiteness argument.  Essentially, we show that applying the Tate construction collapses enough of the target to produce an isomorphism.  We include the following proposition as an easy precursor to the next one; it is not required in future arguments.
\begin{cor}
Let $U$ be a unipotent algebraic group, and $X$ a quasicompact algebraic space with a $U$-action.  Then, the natural map induced by pullback
$$\begin{tikzcd} \OO(\LL^u(X/U))^{\Tate} \arrow[r, "\simeq"] & \OO(\LLf(X/U))^{\Tate}\end{tikzcd}$$
is an isomorphism.  In particular, taking $X = \pt$,
$$\begin{tikzcd} \OO(\LL^u(BU))^{\Tate} = \OO(U/U)^{\Tate} \arrow[r, "\simeq"] & \OO(\LLf(BU)) \simeq k((u))\end{tikzcd}$$
is an isomorphism.
\end{cor}
\begin{proof}
By Lemma \ref{GradedIso}, the map is a pro-graded isomorphism.  Furthermore, applying the Tate construction, we have that $\OO(\LLf(X/U)) \simeq H^\bullet(X; k)((u))$ since $U$ is contractible, where $u$ has cohomological degree 2 and weight 1.  The statement follows from Lemma \ref{ProgradedFinite}: for quasicompact algebraic spaces $X$, $H^\bullet(X; k)$ is finite-dimensional and therefore has has a homogeneous basis.
\end{proof}

A tweaking of the above argument gives us the reductive case.
\begin{cor}
Let $G$ be a reductive algebraic group, and $X$ a quasicompact algebraic space with a $G$-action.  Then, the natural map induced by pullback
$$\begin{tikzcd} \OO(\LL^u(X/G))^{\Tate} \arrow[r, "\simeq"] & \OO(\LLf(X/G))^{\Tate}\end{tikzcd}$$
is an isomorphism.
\end{cor}
\begin{proof}
Let $\mf{h}$ denote the universal Cartan of the Lie algebra $\mf{g}$ of $G$, and $W$ the universal Weyl group acting on $\mf{h}$.  By Proposition \ref{StackCoh} we have
$$\OO(\LLf(BG))^{\Tate} \simeq C_{dR}^\bullet(BG; k) \otimes^! ((u)) \simeq k[[\mf{h}]]^W((u))$$ 
where $\h$ is in cohomological degree zero and subcomplex $H^\bullet(BG; k)$ is given by $k[\mf{h}u]^W$.  By Proposition \ref{UnipotentLoopsQuotient}, and the identification $G//G \simeq T//W$ for reductive groups $G$, we have that $\LL^u(BG) = \LL(BG) \times_{T//W} \widehat{\{0\}}$, and in particular by base change,
$$\OO(\LL^u(BG)) \simeq \OO(\widehat{\{0\}}$$
so that $\OO(\LL^u(BG))^{\Tate} \simeq \OO(\LLf(BG))^{\Tate} \simeq k[[\h]]^W((u))$.  In particular,
$$\OO(\LL^u(X/G))^{\Tate} \rightarrow \OO(\LLf(X/G))^{\Tate}$$
is a map of module objects over the algebra object $k[[\h]]^W((u))$ in the category $\QCoh(\Spec(k^\Z))$, where $\h$ has weight -1 and cohomological degree 0, and its cohomology groups are linear over $H^0(k[[\h]]^W((u))) = k[[\h]]^W$.  We claim that the cohomology groups of the target
$$H^i(\OO(\LLf(X/G))^{\Tate}) \simeq H^i(C_{dR}^\bullet(X/G; k) \otimes^! k((u)))$$ 
are finitely generated over $k[[\h]]^W$ by weight-homogeneous generators; assuming the claim, the result follows from Lemma \ref{CompleteWeights} and Lemma \ref{ProgradedFinite}.

To see the claim, we trace through the identification of Proposition \ref{StackCoh}, keeping track of the weights.  The cotangent complex $\mathbb{L}_X$ has weight -1 and $u$ has weight 1 by convention.  Thus, we find that under the identification of Proposition \ref{StackCoh}, $C^\bullet_{dR}(X/G; k)$ has weight 0 and $u$ has weight 1.  Therefore, it suffices to show that the ring $H^\bullet(X/G; k)$ is finitely generated over $H^\bullet(BG; k)$ by finitely many cohomologically homogeneous generators.   
This can be observed via the derived Cartan model for equivariant cohomology, which computes $H^\bullet(X/G; k)$ via a double complex whose $E_1$ page is $H^\bullet(X) \otimes_k k[[\mf{g}]]^G$, completing the claim.  
\end{proof}

We can now prove Theorem \ref{FormalUni}, which we restate for convenience.
\begin{thm}\label{FormalUniProof}
Let $X$ be a quasicompact algebraic space with an action of an affine algebraic group $G$.  The natural map induced by pullback
$$\begin{tikzcd} \OO(\LL^u(X/G))^{\Tate} \arrow[r, "\simeq"] & \OO(\LLf(X/G))^{\Tate}\end{tikzcd}$$
is an isomorphism.
\end{thm}
\begin{proof}
Every affine algebraic group $G$ embeds as a subgroup of a reductive group $K$.  Apply the previous corollary to $(X \times^G K)/K$.  Note that $X \times^G K$ is not guaranteed to be a scheme, but is always an algebraic space.
\end{proof}

\end{document}